\documentclass[11pt]{article}
\usepackage{amsfonts,amssymb,amsmath,amscd,amsthm}
\usepackage{mathrsfs,mathtools,bbm}
\usepackage{latexsym}
\usepackage{color}
\usepackage{verbatim,enumitem}
\usepackage{hyperref}
\input diagxy
\usepackage[all,cmtip]{xy}
\usepackage{tikz}
\usetikzlibrary{matrix} 
\usepackage[total={155mm,225mm}]{geometry}
\usepackage{geometry} 

\DeclareMathOperator{\cl}{cl}

\DeclareMathOperator{\pt}{pt}
\DeclareMathOperator{\idl}{Idl}
\DeclareMathOperator{\dB}{dB}
\DeclareMathOperator{\Ba}{B}
\DeclareMathOperator{\dSpec}{dSpec}
\DeclareMathOperator{\Spec}{Spec}
\DeclareMathOperator{\dClop}{dClop}
\DeclareMathOperator{\Clop}{Clop}
\DeclareMathOperator{\dO}{d\mathcal{O}}
\DeclareMathOperator{\dpt}{dpt}

\theoremstyle{plain}
\newtheorem{thm}{Theorem}[section]
\newtheorem{lem}[thm]{Lemma}
\newtheorem{prop}[thm]{Proposition}
\newtheorem{cor}[thm]{Corollary}
\theoremstyle{definition}
\newtheorem{defn}[thm]{Definition}

\newtheorem{exmp}[thm]{Example}
\newtheorem{rem}[thm]{Remark}

\newcommand{\bbB}{\mathbb{B}}

\newcommand{\CT}{\mathcal{T}}
\newcommand{\CO}{\mathcal{O}}

\newcommand{\dbt}{{t\!t}}
\newcommand{\dbf}{{f\!\!f}}

\newcommand{\dLat}{\bf dLat}
\newcommand{\dFrm}{\bf dFrm}
\newcommand{\dBoo}{\bf dBA}
\newcommand{\lra}{\longrightarrow}

\begin{document}
  
\title{d-Boolean algebras and their bitopological representation\thanks{This work is supported by the National Natural Science Foundation of China (No. 12371463).}}
\author{Hang Yang, Dexue Zhang\\ {\normalsize School of Mathematics, Sichuan University, Chengdu, China}\\ {\normalsize Email: yanghangscu@qq.com, dxzhang@scu.edu.cn}}
\date{}
\maketitle

\begin{abstract} We present a Stone  duality for bitopological spaces in analogy to the duality between Stone spaces and Boolean algebras, in the same vein as the duality between d-sober bitopological spaces and spatial d-frames established by Jung and Moshier. Precisely, we introduce the notion of d-Boolean algebras and prove that the category of such algebras is dually equivalent to the category of   compact and zero-dimensional bitopological spaces satisfying the $T_0$ separation axiom. \vskip 2pt 


\end{abstract}

\section{Introduction} 
By topologizing the set of prime ideals (or  the set of prime filters) of a Boolean algebra  Marshall Stone   \cite{Stone36,Stone37a} established a duality  between Boolean algebras and   compact and zero- dimensional Hausdorff spaces (which are now known as Boolean spaces or Stone spaces).  
$$ \bfig
\morphism(0,0)|a|/@{->}@<3pt>/<660,0>[{\bf BA}^{\rm op}`{\bf BoolSp};\Spec ]
\morphism(0,0)|b|/@{<-}@<-3pt>/<660,0>[{\bf BA}^{\rm op}`{\bf BoolSp}; \Clop] \efig$$
This duality has a far-reaching influence (see  \cite[Introduction]{Johnstone}) and has led to the discovery of many dualities in mathematics. One example is the duality 
$$\bfig
\morphism(0,0)|a|/@{->}@<3pt>/<700,0>[{\bf SpFrm}^{\rm op}`{\bf SobTop};\pt]
\morphism(0,0)|b|/@{<-}@<-3pt>/<700,0>[{\bf SpFrm}^{\rm op}`{\bf SobTop};\CO]  
\efig$$
between  spatial frames and sober topological spaces, see e.g. \cite{Johnstone,PP2012}.

The duality between Boolean algebras and Stone spaces is closely related to the duality between spatial frames and sober spaces. 
First, the duality between spatial frames and sober spaces cuts down to a dual equivalence between compact and zero-dimensional frames and Stone spaces. Second, assigning  to each Boolean algebra the frame of its ideals defines an equivalence between  Boolean algebras and  compact and zero-dimensional frames. The latter is a point-free version of the Stone representation of Boolean algebras. The dual equivalence between Boolean algebras and Stone spaces is then the  composite of these two equivalences.  
  
The purpose of this paper is to establish an analogue of this duality   in the realm of bitopological spaces. 

In 2006 Jung and Moshier \cite{JM2006} introduced   d-frames as algebraic duals of bitopological spaces. The relationship between d-frames and bitopological spaces is parallel to that between frames and topological spaces, see e.g.  \cite{Jakl,JM2006,Klinke}. In particular,  Jung and Moshier established a duality between spatial d-frames and d-sober bitopological spaces: 
$$\bfig\morphism(0,0)|a|/@{->}@<3pt>/<800,0>[{\bf SpdFrm}^{\rm op}`{\bf SobBiTop};\dpt] \morphism(0,0)|b|/@{<-}@<-3pt>/<800,0>[{\bf SpdFrm}^{\rm op}`{\bf SobBiTop};\dO] \efig$$
 
The duality we wish to establish is in the same vein as the duality of Jung and Moshier. Precisely, we wish to present a dual equivalence $$\bfig
\morphism(0,0)|a|/@{->}@<3pt>/<750,0>[{\bf dBA}^{\rm op}`{\bf dBoolSp};\dSpec ]
\morphism(0,0)|b|/@{<-}@<-3pt>/<750,0>[{\bf dBA}^{\rm op}`{\bf dBoolSp}; \dClop] \efig$$ between  an analogue of the category of Boolean algebras and an analogue of the category of Stone spaces. For {\bf dBoolSp} we take the category of d-Boolean spaces (see Section \ref{Stone bitopological spaces} for definition), these spaces are called compact totally order disconnected spaces in \cite{JM2006} and   pairwise Stone spaces in \cite{BBGK}. Our main task is to find the category {\bf dBA}, of which the objects are called d-Boolean algebras. In a nutshell, d-Boolean algebras are to bitopological spaces what Boolean algebras  are to topological spaces. The duality between d-Boolean algebras and d-Boolean spaces will be established by making the set of prime d-filters of a d-Boolean algebra into a bitopological space in a way parallel to that in the classical case. 

A d-Boolean algebra is defined to be a d-complemented d-lattice (Definition \ref{defn of d-boolean}). The category of d-Boolean algebras will be shown to be equivalent to the category of distributive lattices (Proposition \ref{DLat and DBL}). So, the duality concerned in this paper is also related to topological representation of  distributive lattices. In 1937, Stone \cite{Stone37b} proved that the category {\bf DisLat} of distributive lattices is dually equivalent to the category {\bf CohSp} of coherent spaces and coherent maps (see \cite[page 66]{Johnstone}). In 1970, Priestley \cite{Priestley70,Priestley72} proved that {\bf DisLat} is dually equivalent to the category {\bf Pries} of Priestley spaces  (also known as ordered Stone spaces) and order-preserving continuous maps (see \cite[page 75]{Johnstone}). In 2006, Jung and Moshier \cite[Corollary 8.13]{JM2006} proved that  {\bf DisLat} is dually equivalent to the category of d-Boolean spaces and continuous maps; also see \cite[Theorem 5.3]{BBGK} for this duality. 
Therefore, the category {\bf CohSp} of coherent spaces, the category {\bf Pries} of Priestley spaces and the category {\bf dBoolSp} of d-Boolean spaces are equivalent to each other, all of them are of a topological nature. Are there natural categories of an algebraic nature that are equivalent to the category of distributive lattices? In 2017, Jakl \cite[Section 2.6]{Jakl} demonstrated that the category of compact and d-zero-dimensional d-frames is an example of such categories. The category of d-Boolean algebras provides another such example. 

The contents are arranged as follows. In Section \ref{review}  we briefly review  the duality between Boolean algebras and Stone spaces. In Section \ref{Stone bitopological spaces} we recall  some basic concepts of bitopological spaces and introduce  the notion of d-Boolean spaces. In Section \ref{dBA-section} d-Boolean algebras are postulated as a special kind of d-lattices introduced in \cite{Klinke,KJM}. Structures of d-lattices and d-Boolean algebras are examined. It is proved that the category of d-Boolean algebras is equivalent to the category of compact and d-zero-dimensional d-frames (Theorem  \ref{dBA vs dfrm}); and that the category of d-lattices is equivalent to the category of  coherent d-frames and coherent d-frame homomorphisms (Corollary \ref{d-lattice vs d-frame}). In Section \ref{d-spectrum} we introduce  the notion of prime d-filters for d-lattices and make  the set of prime d-filters of a d-lattice  into a bitopological space, called the spectrum of the d-lattice. We prove  that the spectrum of each d-Boolean algebra is a d-Boolean space and every d-Boolean space is the spectrum of a unique (up to isomorphism) d-Boolean algebra, arriving at the desired duality.  
  
\section{Review of the Stone representation of Boolean algebras}  \label{review}
In this section, we briefly review the duality between Boolean algebras and Stone spaces established by Stone \cite{Stone36,Stone37a}, for comprehensive accounts of the story we refer to the monograph of Johnstone   \cite{Johnstone}, and the monograph of Gehrke and van Gool \cite{GG2024}. Our reference for category theory is \cite{AHS}, for domain theory is \cite{Gierz2003}, and for general topology is \cite{Engelking}.

Throughout this paper, by a distributive lattice we always mean a bounded one; that means, a distributive lattice has top element $1$ and a bottom element $0$.  
An element $a$ of a distributive lattice $(L,\sqsubseteq)$ is \emph{complemented} if there is some $b\in L$ such that $a\sqcup b=1$ and $a\sqcap b=0$. Such $b$,  when exists, is necessarily unique and is called the \emph{complement} of $a$. A Boolean algebra is a distributive lattice of which all elements are complemented.  

The clopen (closed and open) sets of a topological space form a Boolean algebra. Sending each topological space to this algebra gives rise to a contravariant functor  $$\Clop\colon{\bf Top}\longrightarrow{\bf BA}^{\rm op} $$  from the category of topological spaces to the category of Boolean algebras.  The celebrated Stone representation of Boolean algebras says that restricting the domain of the functor $\Clop$ to the subcategory of Boolean spaces yields an equivalence of categories.  

A \emph{Boolean space}, also known as a \emph{Stone space}, is a compact topological space $X$ that satisfies one, hence all, of the following equivalent conditions  \cite[page 70]{Johnstone}:
\begin{enumerate}[label=\rm(\roman*)] \setlength{\itemsep}{0pt} 
\item $X$ is $T_0$ and zero-dimensional in the sense that clopen  sets of $X$ form a base for the topology.
\item $X$ is totally separated in the sense that whenever $x$ and $y$ are distinct points of $X$, there is a clopen  set of $X$ containing $x$ but not $y$.
\item $X$ is Hausdorff and totally disconnected in the sense that  connected subsets of $X$ are single points.
\end{enumerate} The full subcategory of {\bf Top} composed of Boolean spaces is denoted by {\bf BoolSp}. 

An \emph{ideal} of a distributive lattice $L$ is a non-empty subset $I$ of $L$ such that $a\in I$ and $b\sqsubseteq a$ implies $b\in I$, and that $a,b\in I$ implies $a\sqcup b\in I$. The set of all ideals of $L$ is denoted by $\idl L$.  For an ideal $I$ of $L$, we say that \begin{itemize} \setlength{\itemsep}{0pt} \item $I$ is \emph{proper} if $1\notin I$. \item  $I$ is \emph{prime} if it is proper and  $a\sqcap b\in I\implies a\in I$ or $ b\in I$. \end{itemize}

The notions \emph{filter}, \emph{proper filter} and \emph{prime filter} are defined dually. In any distributive lattice, the complement of a prime ideal is a prime filter and vice versa.

For each distributive lattice $L$, let $\Spec L$ be the set of all prime filters of $L$. The collection  of subsets $$\Phi(I)=\{F\in\Spec L: I\cap F\not=\emptyset\},\quad I\in\idl L $$ is a topology on $\Spec L$, the resulting topological space is called the  \emph{spectrum}  of $L$. The spectrum of each Boolean algebra is a Boolean space, so we have a  contravariant functor  $$\Spec\colon{\bf BA}^{\rm op}\longrightarrow{\bf BoolSp}.$$ The functors $\Clop\colon{\bf BoolSp}\longrightarrow{\bf BA}^{\rm op}$ and $\Spec\colon{\bf BA}^{\rm op}\longrightarrow{\bf BoolSp}$ witness that the category  of Boolean algebras is dually equivalent to the category of Boolean spaces.   

\section{d-Boolean spaces} \label{Stone bitopological spaces}
A \emph{bitopological space} \cite{Kelly} is a triple $(X,\tau_+,\tau_-)$, where $X$ is a set, $\tau_+$ and $\tau_-$ are topologies on $X$. A  \emph{continuous}  (also called bicontinuous in the literature) map
$$f\colon (X,\tau_+,\tau_-) \longrightarrow (X',\tau_+',\tau_-')$$ is a map $f\colon X\longrightarrow X'$ such that both $f\colon (X,\tau_+)\longrightarrow (X',\tau_+')$ and $f\colon (X,\tau_-)\longrightarrow (X',\tau_-')$ are continuous. The category of bitopological spaces and continuous maps is denoted by  $${\bf BiTop}.$$ 

The assignment $(X,\CT)\mapsto (X,\CT,\CT)$ defines a full and faithful functor $$\omega\colon {\bf Top}\lra{\bf BiTop}.$$ The functor $\omega$ embeds the category of topological spaces as a simultaneously reflective and coreflective full subcategory of the category of bitopological spaces; the {\bf Top}-coreflection of a bitopological space $(X,\tau_+, \tau_-)$ is given by the topological space $(X,\tau_+\vee\tau_-)$ \cite{HZ2025}.  

\begin{defn}\label{compact bitop}
A bitopological space $(X,\tau_+,\tau_-)$ is
\begin{enumerate}[label=\rm(\roman*)] \setlength{\itemsep}{0pt}
\item (\cite{Mu})  $T_0$ if for each pair of distinct points of $X$, there is a $\tau_+$-open set or a
$\tau_-$-open set containing one of the points, but not the other.
\item (\cite{Swart}) \label{join compactness}  compact if every subset of $\tau_+\cup\tau_-$  covering $X$ has a finite subset that covers $X$.
 
\item \label{zero-dimensionality} (\cite{Reilly}) zero-dimensional (called pairwise zero-dimensional in \cite{Reilly}) if $\tau_{+}$ has a base of $\tau_{-}$-closed sets and $\tau_{-}$ has a base of $\tau_{+}$-closed sets. 
\end{enumerate}
\end{defn}
 
Let $(X,\tau_+,\tau_-)$  be a bitopological space. Then, $(X,\tau_+,\tau_-)$ is $T_0$  if and only if it is join $T_0$ \cite{Lal} in the sense that the topological space $(X,\tau_{+}\vee\tau_{-})$ is $T_0$;  and, under the assumption of the Axiom of Choice,  $(X,\tau_+,\tau_-)$ is compact if and only if it is join compact in the sense that the topological space $(X, \tau_{+}\vee \tau_{-})$ is compact.  

For each topological space $X$, the bitopological space $\omega(X)$ is $T_0$,   compact, and zero-dimensional if and only if so is $X$  (as a topological space),  respectively.
  
For each bitopological space $(X,\tau_+,\tau_-)$, we write $\leq_+$ and $\leq_-$ for the specialization order of the topological spaces $(X,\tau_+)$ and $(X,\tau_-)$, respectively; and let $\leq$ be the intersection of $\leq_+$ and the opposite of $\leq_-$; that is, $$\leq\thinspace = \thinspace\leq_+\cap \geq_-.$$ 

\begin{defn}\label{order separated}  (\cite[Definition 3.7] {JM2008})
A  bitopological space $(X;\tau_{+},\tau_{-})$ is order-separated provided that the binary relation $\leq$ (which is $\leq_+\cap \geq_-$) is a partial order, and  that if $x\not\leq  y$ then there exist a $\tau_+$-neighborhood $U$ of $x$ and a $\tau_-$-neighborhood $V$ of $y$ such that  $U\cap V=\emptyset$.   
\end{defn}

For each topological space $X$, the bitopological space $\omega(X)$ is order-separated if and only if  $X$ is Hausdorff. Lemma 3.8 of Jung and Moshier \cite{JM2008} says that in an order-separated bitopological space $(X,\tau_+,\tau_-)$, the specialization order $\leq_+$ of $(X,\tau_+)$ is opposite to the specialization order $\leq_-$ of  $(X,\tau_-)$,   hence $\leq\thinspace=\thinspace\leq_+= \thinspace \geq_-$.
 
\begin{defn}
A  bitopological space $(X,\tau_+,\tau_-)$ is totally order-separated provided that the binary relation $\leq$  is a partial order, and that if $x\not\leq y$ then there is a $\tau_+$-open  and $\tau_-$-closed set  containing $x$ but not $y$. 
\end{defn}

A totally order-separated bitopological space is clearly order-separated. So, a totally order-separated bitopological space is exactly a totally order-disconnected space in the sense of \cite[Definition 8.12]{JM2006}.  The term \emph{totally order-separated} is chosen simply because for each topological space $X$, the bitopological space $\omega(X)$ is totally order-separated if and only if  $X$ is totally separated in the sense of \cite[page 69]{Johnstone}.

\begin{prop}\label{Stone bitopological space}
For each bitopological space $(X,\tau_+,\tau_-)$, the following  are equivalent:
\begin{enumerate}[label=\rm(\arabic*)] \setlength{\itemsep}{0pt} 
\item $(X,\tau_+,\tau_-)$ is  $T_0$, compact and zero-dimensional. 
\item $(X,\tau_+,\tau_-)$ is compact and totally order-separated.
\end{enumerate}
\end{prop}

\begin{proof}
$(1)\Rightarrow(2)$ We proceed with two steps.

{\bf Step 1}.   $(X,\tau_+,\tau_-)$ is order-separated. 

First  we show that $\leq_+  =\thinspace \geq_-$. This equality is the content of Lemma 3.3 in \cite{BBGK}. We include a proof here for convenience of the reader. If $x\not\leq_+y$,  by zero-dimensionality there is a $\tau_+$-open  and $\tau_-$-closed set $U$ containing $x$ but not $y$. Then $X\setminus U$ is a $\tau_-$-open set containing $y$ but not $x$, from which one infers that $y\not\leq_- x$. Likewise, $y\not\leq_-x$ implies  $x\not\leq_+y$. Therefore $\leq_+  =\thinspace  \geq_-$. 

Next we show that $\leq$ is a partial order. For this we show that  $\leq_+$ is a partial order. Let $x,y$ be a pair of distinct points of $X$. Since   $(X,\tau_{+}, \tau_{-})$ is $T_0$, there is some $U\in\tau_+\cup\tau_-$ that contains one of the points, but not the other, so, either $x\not\leq_+ y$, or  $y\not\leq_+ x$, or  $x\not\leq_- y$, or  $y\not\leq_- x$. Since $\leq_+  =\thinspace \geq_-$, then either $x\not\leq_+ y$ or  $y\not\leq_+ x$, hence $\leq_+$ is a partial order.

Now we show that if $x\not\leq y$ then there exist a $\tau_+$-neighborhood $U$ of $x$ and a $\tau_-$-neighborhood $V$ of $y$ such that  $U\cap V=\emptyset$.  Since $\leq\thinspace =\thinspace \leq_+$,  there is a $\tau_+$-open set $W$ containing $x$ but not $y$. By zero-dimensionality there is a $\tau_+$-open  and $\tau_-$-closed set $U$ such that $x\in U\subseteq W$. Then $U$ is a $\tau_+$-neighborhood  of $x$  and $V\coloneqq X\setminus U$ is a $\tau_-$-neighborhood of $y$ such that  $U\cap V=\emptyset$.

{\bf Step 2}.  $(X,\tau_+,\tau_-)$ is totally order-separated. 

It suffices to check that if $x\not\leq y$, then there is a $\tau_+$-open  and $\tau_-$-closed set containing $x$ but not $y$.  Since $(X,\tau_+,\tau_-)$ is order-separated, there is a $\tau_+$
open set $U$ containing $x$ but not $y$. Since $(X,\tau_+,\tau_-)$ is zero-dimensional, there is a $\tau_+$-open  and $\tau_-$-closed set $V$ such that $x\in V\subseteq U$. Thus $V$ is a $\tau_+$-open  and $\tau_-$-closed set containing $x$ but not $y$. 
 
$(2)\Rightarrow(1)$   It suffices to show that $(X,\tau_+,\tau_-)$ is zero-dimensional. Let $x\in X$ and $U$ be a $\tau_+$-open set containing $x$. We find a $\tau_+$-open and  $\tau_-$-closed set $V$ such that $x\in V\subseteq U$. For each $y\in X\setminus U$, we have $x\not\leq y$, so there is a $\tau_+$-open  and $\tau_-$-closed set $U_y$ containing $x$ but not $y$. Choosing one such $U_y$ for each $y\in X\setminus U$ we obtain a $\tau_-$-open cover $\{X\setminus U_y:y\in X\setminus U\}$ of $X\setminus U$. By  compactness of $(X,\tau_+,\tau_-)$, $X\setminus U$ can be covered by finitely many elements of $\{X\setminus U_y:y\in X\setminus U\}$, say, $X\setminus U_1, X\setminus U_2,\cdots, X\setminus U_n$. Let $V=\bigcap_{i\leq n}  U_i$. Then $V$ satisfies the requirement. This shows that $\tau_+$ has a base of $\tau_-$-closed sets.
Likewise, $\tau_-$ has a base of $\tau_+$-closed sets. Therefore  $(X,\tau_+,\tau_-)$ is zero-dimensional.
\end{proof} 

In a $T_0$ and zero-dimensional  bitopological space,  connected sets (in the sense of Pervin \cite{Pervin}) are single points. We do not know whether the requirement being zero-dimensional in Proposition \ref{Stone bitopological space}\thinspace(1) can be weakened to that connected sets are single points.

A bitopological space satisfying the equivalent conditions of Proposition \ref{Stone bitopological space} is called a  \emph{d-Boolean space}, or   a \emph{d-Boolean bitopological space}. The category of d-Boolean spaces and continuous maps is denoted by {\bf dBoolSp}. 
It is clear that a topological space $X$ is a Boolean space if and only if the bitopological space $\omega(X)$ is a d-Boolean space.  

\begin{rem} The argument of Proposition \ref{Stone bitopological space} shows that if $(X,\tau_+,\tau_-)$ is a d-Boolean space, then it is bi-$T_0$ in the sense that both $(X,\tau_+)$ and $(X, \tau_-)$ are $T_0$ topological spaces. So, by Lemma 2.5 of \cite{BBGK} one sees that d-Boolean spaces are precisely  \emph{pairwise Stone spaces} in  \cite[Definition 2.10]{BBGK}. \end{rem}

\section{d-Lattices and d-Boolean algebras} \label{dBA-section}

Suppose   $(L,\sqsubseteq,0,1)$ is a  distributive lattice; $\dbt$ and $\dbf$ are a  pair of complementary  elements, that means, $\dbt\sqcup\dbf=1$ and $\dbt\sqcap\dbf=0$. Let $L_+$ be the lower set $\downarrow\!\dbt$ and let $L_-$ be the lower set $\downarrow\!\dbf$ of $L$. Then the  assignment $$a\mapsto (a\sqcap\dbt,a\sqcap\dbf)$$ is an order isomorphism from $L$ to the product lattice $L_+\times L_-$, the inverse  isomorphism takes each element $(a,b)$ of $L_+\times L_-$ to the join $a\sqcup b$ in $(L,\sqsubseteq)$. 

We use the isomorphism $L\cong L_+\times L_-$ to define a new order $\leq$ on $L$ as follows: $$(a_1,b_1)\leq (a_2,b_2)\iff a_1\sqsubseteq a_2,~b_1\sqsupseteq b_2.$$ Then $(L,\leq)$ is a distributive lattice with $\dbt$ as top element and $\dbf$ as bottom element. The meet $x\wedge y$ and the join $x\vee y$ in $(L,\leq)$ are computed in terms of the meet and join operations of $(L,\sqsubseteq)$ as follows:
$$x\wedge y=(x\sqcap \dbf)\sqcup (y\sqcap \dbf)\sqcup (x\sqcap y);$$
$$x\vee y=(x\sqcap \dbt)\sqcup(y\sqcap \dbt)\sqcup(x\sqcap y).$$ The lattice structure $(L,\wedge,\vee)$ is already known in  \cite[page 751]{BK47}. Following Jung and Moshier \cite{JM2006, JM2008} we call $\sqsubseteq$ and $\leq$ the \emph{information order} and the \emph{logic order} of $(L;\dbt,\dbf)$, respectively.  
If $\{\dbt,\dbf\}=\{1,0\}$, then the logic order $\leq$ coincides with the order $\sqsubseteq$ or its opposite. So, in order to avoid degeneracy, in this paper we always assume that the complementary pair $\{\dbt,\dbf\}$ is different from $\{1,0\}$.
 
Let $\mathbb{B}$ be the Boolean algebra $\{0,1,\dbt,\dbf\}$, with $0$ being the bottom element, $1$ being the top element, $\dbt$ and $\dbf$ being complements of each other, as visualized below: \[\bfig \morphism(0,0)/@{-}/<-250,-250>[1`\dbf;]  \morphism(0,0)/@{-}/<250,-250>[1`\dbt;] \morphism(-250,-250)/@{-}/<250,-250>[\dbf`0;]  \morphism(250,-250)/@{-}/<-250,-250>[\dbt`0;] \efig\] With the logic order $\mathbb{B}$ becomes \[\bfig \morphism(0,0)/@{-}/<-250,-250>[\dbt`1;]  \morphism(0,0)/@{-}/<250,-250>[\dbt`0;] \morphism(-250,-250)/@{-}/<250,-250>[1`\dbf;]  \morphism(250,-250)/@{-}/<-250,-250>[0`\dbf;] \efig\]

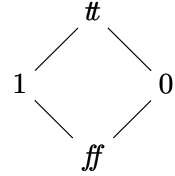
\begin{defn} (\cite{Klinke,KJM})
A  d-lattice is a structure $(L;\dbt,\dbf;{\bf con},{\bf tot})$, where  $L$ is a  distributive lattice, $\dbt$ and $\dbf$ are a pair of  complementary elements of $L$,  {\bf con}  and {\bf tot}  are subsets of $L$ (called the  consistency predicate  and the  totality  predicate, respectively), subject to the following conditions:
\begin{itemize} \setlength{\itemsep}{0pt}		
\item  $\dbt,\dbf\in{\bf con}$; 

\item   $\dbt,\dbf\in{\bf tot}$;
			
\item {\bf con} is a lower set with respect to the information order $\sqsubseteq$;  

\item   {\bf tot} is an upper set with respect to the information order $\sqsubseteq$;  

\item   {\bf con} and {\bf tot} are sublattices of $L$ under the logic order $\leq$, i.e., sublattices of  $(L,\leq,\wedge,\vee)$;  
 
\item ({\bf con}--{\bf tot}) $\alpha\in{\bf con}$, $\beta\in{\bf tot}$,  $(\alpha\sqcap\dbt=\beta\sqcap\dbt\text{ or }\alpha\sqcap\dbf= \beta\sqcap\dbf)\implies \alpha\sqsubseteq\beta$.
\end{itemize} 

A d-lattice homomorphism $f\colon\mathcal{L}\longrightarrow\mathcal{L}'$ between d-lattices  
is a lattice homomorphism $f\colon L\longrightarrow L'$ that preserves $\dbt$, $\dbf$, ${\bf con}$ and ${\bf tot}$ in the sense that
$$f(\dbt)=\dbt',\quad f(\dbf)=\dbf',\quad f({\bf con})\subseteq{\bf con}',\quad f({\bf tot})\subseteq{\bf tot}'.$$
The category of d-lattices and d-lattice homomorphisms is denoted by  $\dLat.$  \end{defn}

\begin{exmp} \label{B as d-lattice} \begin{enumerate}[label=\rm(\roman*)] \setlength{\itemsep}{0pt}
\item There is a unique way to make the Boolean algebra $\mathbb{B}=\{0,1,\dbt,\dbf\}$ into a d-lattice; that is, ${\bf con}=\{0,\dbt,\dbf\}$ and ${\bf tot}=\{1,\dbt,\dbf\}$. In this paper we always view $\mathbb{B}$ as a d-lattice (in this way). 
\item Let $L$ be a distributive lattice with $\dbt,\dbf$ being a complementary pair of elements. Then $(L;\dbt,\dbf;{\bf con},{\bf tot})$  is a d-lattice, where   ${\bf con}= \thinspace\downarrow\!\dbt \thinspace \cup\! \downarrow\!\dbf$ and ${\bf tot}= \thinspace\uparrow\!\dbt\thinspace\cup\! \uparrow\!\dbf$. 
\end{enumerate} \end{exmp}

\begin{exmp} \label{order dual} (\cite[Definition 2.1.4]{Klinke}) For each  d-lattice $\mathcal{L}=(L;\dbt,\dbf;{\bf con},{\bf tot})$, the structure $$\mathcal{L}^\partial\coloneqq (L^{\rm op};\dbt^\partial,\dbf^\partial; {\bf con}^\partial, {\bf tot}^\partial)$$ is a d-lattice, where $L^{\rm op}$ is the opposite of $L$, and \begin{itemize} \setlength{\itemsep}{0pt}
 \item $\dbt^\partial=\dbf$, \item $\dbf^\partial=\dbt$, \item  ${\bf con}^\partial={\bf tot}$, \item ${\bf tot}^\partial={\bf con}$. \end{itemize} The d-lattice $\mathcal{L}^\partial$ is  called the \emph{order-dual}  of $\mathcal{L}$. It is clear that $\mathcal{L}^{\partial\partial} =\mathcal{L}$. \end{exmp}
  
\begin{rem}For each d-lattice $\mathcal{L}=(L;\dbt,\dbf;{\bf con},{\bf tot})$, the underlying lattice $L$ is isomorphic to the product lattice $L_+\times L_-$, where $L_+$ and $L_-$ are the principal lower sets $\downarrow\!\dbt$ and  $\downarrow\!\dbf$ of $L$, respectively. So, a d-lattice can be equivalently presented as a pair $L_+,L_-$ of distributive lattices together with two subsets of $L_+\times L_-$ subject to certain requirement, as in \cite[Definition 1]{KJM} and \cite[Definition 2.1.2]{Klinke}. This presentation is very convenient, especially in the construction of d-lattices, see e.g. Example \ref{omegaH} and Example \ref{dO} below. Both presentations are used in this paper. In order to switch between these presentations without causing confusion, in this paper we identify $L$ with the product lattice $L_+\times L_-$,  and in particular, we identify the lower sets $L_+$ and $L_-$ of $L$ with the subsets $\{(a,0):a\sqsubseteq\dbt\}$ and $\{(0,b):b\sqsubseteq\dbf\}$ of the product lattice $L_+\times L_-$, respectively. 
\begin{center}\begin{tikzpicture}[scale=0.6]  
\fill [lightgray] (0,0) rectangle (5,4); 
\draw[thick] (0,0) -- (5,0); 
\draw[thick] (0,0) -- (0,4);  
\node at (2.5,-0.4) {{\small $L_+$}}; \node at (5,-0.4) {{\small $\dbt$}}; 
\node at (-0.5,2) {{\small $L_-$}};  \node at (-0.4,4) {{\small $\dbf$}}; \node at (-0.3,-0.3) {{\small $0$}}; \node at (5.3,4) {{\small $1$}}; 
\draw [fill] (0,4) circle [radius=0.05]; \draw [fill] (5,0) circle [radius=0.05]; \node at (2.5,-1.5) {{\small $L\cong L_+\times L_-$}};
\end{tikzpicture}\end{center} \end{rem}

\begin{exmp}\label{omegaH} Let $(L,\sqsubseteq,\sqcap,\sqcup)$ be a  distributive lattice. Then,  $L\times L$ with the pointwise order (also denoted by $\sqsubseteq$) is a distributive lattice with $(1,0)$ and $(0,1)$ being a complementary  pair.  The structure  $\omega(L)\coloneqq(L\times L;\dbt,\dbf;{\bf con},{\bf tot})$  is a d-lattice, where  \begin{itemize} \setlength{\itemsep}{0pt} 
\item $\dbt=(1,0)$, $\dbf=(0,1)$; \item ${\bf con}=\{(a,b)\in L\times L: a\sqcap b=0\}$;  \item ${\bf tot}=\{(a,b)\in L\times L: a\sqcup b=1\}$. \end{itemize} We only need to check the ({\bf con}--{\bf tot}) condition.  Assume, without loss of generality, that $(a_1,b_1)\in{\bf con}$, $(a_2,b_2)\in{\bf tot}$, and   $(a_1,b_1)\sqcap(1,0)=(a_2,b_2)\sqcap(1,0)$. Since  meets and joins in $(L\times L,\sqsubseteq)$ are computed pointwise, it follows that  $a_1=a_2$ and $$b_1=b_1\sqcap(a_2\sqcup b_2)= (b_1\sqcap a_1)\sqcup (b_1\sqcap b_2)=b_1\sqcap b_2,$$ then $b_1\sqsubseteq b_2$ and $(a_1,b_1)\sqsubseteq(a_2,b_2)$. So, we obtain a full and faithful functor   $$\omega\colon {\bf DisLat}\lra{\bf dLat}$$ from the category {\bf DisLat} of distributive lattices and lattice homomorphisms to the category of d-lattices and d-lattice homomorphisms. The construction of  $\omega(L)$ is a direct generalization of the construction of a d-frame out of a frame in Jung and Moshier \cite[page 47]{JM2006}. 
\end{exmp}
 
Consider the d-lattice $\omega(L)$. Then, $${\bf con}\cap{\bf tot}= \{(a,b)\in L\times L: a\sqcup b=1, a\sqcap b=0\}.$$ In other words, ${\bf con}\cap{\bf tot}$ is the set of all complementary pairs of $L$. Furthermore, if $B$ is the sublattice of $L$ composed of complemented elements, which is readily verified to be a Boolean algebra, then   ${\bf con}\cap{\bf tot}$ is equal to $\{(a,\neg a):a\in B\}$, which is an anti-chain (i.e., a discrete set) with respect to the information order of $L\times L$, and isomorphic to the Boolean algebra $B$ with respect to the logic order.

In Jung and Moshier \cite{JM2006,JM2008}, a d-lattice  $(L;\dbt,\dbf;{\bf con},{\bf tot})$  with $L$ being a frame is called a d-frame. However, following  \cite{JJP,KJM},  we reserve the term \emph{d-frame} for d-lattices  $(L;\dbt,\dbf;{\bf con},{\bf tot})$ for which $L$ is a frame and the consistency predicate {\bf con} is a Scott closed set under the information order. Such d-lattices are called \emph{reasonable d-frames} in   \cite{JM2006,JM2008}.  
A d-frame homomorphism $f:\mathcal{L} \longrightarrow \mathcal{L}'$ between d-frames is a frame homomorphism $f:L\longrightarrow L'$ that preserves $\dbt$, $\dbf$, ${\bf con}$ and ${\bf tot}$. The category of d-frames and d-frame homomorphisms is denoted by   $\bf{dFrm}$.   
For each frame $L$, the d-lattice $\omega(L)$ in Example \ref{omegaH} is a d-frame. This gives rise to a full and faithful functor  $$\omega\colon {\bf Frm}\lra{\bf dFrm}.$$ 
 
 \begin{exmp}\label{dO}(\cite{JM2006,JM2008}) For each bitopological space $(X, \tau_+, \tau_-)$, the structure  $$\dO(X, \tau_+, \tau_-)\coloneqq (L;\dbt,\dbf;{\bf con},{\bf tot})$$  is a d-frame, where  \begin{itemize} \setlength{\itemsep}{0pt}
 \item $L$ is the product frame $\tau_+\times\tau_-$; \item $\dbt=(X,\emptyset)$, $\dbf=(\emptyset,X)$; \item ${\bf con}=\{(U,V)\in \tau_+\times\tau_-: U\cap V=\emptyset\}$;   \item ${\bf tot}=\{(U,V)\in \tau_+\times\tau_-: U\cup V=X\}$. \end{itemize} 
In this way we obtain a contravariant functor $$\dO\colon {\bf BiTop}\lra {\bf dFrm}^{\rm op}.$$  
\end{exmp} 

It is known that the forgetful functor from the category of frames to that of distributive lattices has a left adjoint. As we shall see, so does the forgetful functor $U\colon \dFrm\longrightarrow\dLat$. 

Suppose  $(L;\dbt,\dbf;{\bf con},{\bf tot})$ is a d-lattice. We write $\idl L$ for the frame of ideals of the distributive lattice $L$. Then, the principal ideals $L_+=\thinspace\downarrow\!\dbt$ and $L_-=\thinspace\downarrow\!\dbf$ form a complementary pair of $\idl L$, and $$\idl\mathcal{L}\coloneqq(\idl L;L_+,L_-;{\bf con}_{\idl},{\bf tot}_{\idl})$$
is a d-frame,  where $${\bf con}_{\idl}=\{I\in\idl L: I\subseteq {\bf con}\},\quad {\bf tot}_{\idl}=\{I\in \idl L: I\cap {\bf tot}\neq\emptyset\}.$$ The d-frame $\idl\mathcal{L}$ is called the \emph{d-frame of ideals} of the d-lattice $\mathcal{L}$. In this way we obtain a functor
 $$\idl\colon \dLat\longrightarrow\dFrm.$$  
 
\begin{prop}\label{idl left adjoint to forgetful} The functor $\idl\colon \dLat\longrightarrow\dFrm$  is left adjoint to the  forgetful functor $U\colon \dFrm\longrightarrow\dLat$.  
\end{prop}

\begin{proof} Let 
$\mathcal{L} =(L;\dbt,\dbf; {\bf con},{\bf tot})$ be a d-lattice.  It is clear that the assignment $a\mapsto \thinspace \downarrow\!a$ defines a d-lattice homomorphism $\eta\colon \mathcal{L}\lra \idl\mathcal{L}$. So, it suffices to show that for each d-lattice homomorphism $f\colon\mathcal{L} \longrightarrow\mathcal{M}$ with $\mathcal{M}=(M;\dbt,\dbf;{\bf con},{\bf tot})$ being a d-frame, there is a unique d-frame homomorphism $\overline{f}\colon \idl\mathcal{L}\lra\mathcal{M}$ such that  $f= \overline{f}\circ \eta$.
$$\bfig\qtriangle/->`->`-->/[\mathcal{L}`\idl\mathcal{L}`\mathcal{M};\eta` f`\overline{f}]\efig$$
The map   $\overline{f}\colon\idl L\longrightarrow M$ given by $\overline{f}(I)=\sup_{a\in I} f(a)$  is readily verified to be that unique d-frame homomorphism $\idl\mathcal{L} \lra\mathcal{M}$.
\end{proof}

In order to identify  d-frames  of the form $\idl\mathcal{L}$ for some d-lattice $\mathcal{L}$,  we need some notions. 

An element $x$ of a partially ordered set $P$ is \emph{finite} \cite{Gierz2003,Johnstone} if for each directed subset $D$ of $P$,   $x\sqsubseteq \sup D$ implies that $x\sqsubseteq d$ for some $d\in D$. 
A frame $L$ is \emph{coherent} \cite[page 63]{Johnstone} if \begin{enumerate}[label=\rm(\roman*)] \setlength{\itemsep}{0pt}
\item Every element of $L$ is expressible as a join of finite elements; and 
\item The finite elements form a sublattice of $L$, i.e., 1 is finite, and the meet of two finite elements is finite.
\end{enumerate}

Coherent frames are precisely  frames of ideals of   distributive lattices, see  \cite[page 64]{Johnstone}. 
A d-frame $\mathcal{L}=(L;\dbt,\dbf;{\bf con},{\bf tot})$ is said to be \emph{compact} \cite[Definition 7.5]{JM2008} if the totality predicate {\bf tot} is a Scott open set of $(L,\sqsubseteq)$.  

\begin{lem}\label{t and f are comapct}
Let $\mathcal{L}=(L;\dbt,\dbf;{\bf con},{\bf tot})$ be a compact d-frame. Then both $\dbt$ and $\dbf$ are finite elements of the lattice $(L,\sqsubseteq)$. 
\end{lem}

\begin{proof}
It suffices to show that $\dbt$ is a finite element of  $(L_{+},\sqsubseteq)$ and $\dbf$ is  a finite element of  $(L_{-},\sqsubseteq)$. Suppose $\{a_{i}\}_{i\in I}$ is a directed subset  of  $(L_{+},\sqsubseteq)$ with $\dbt\sqsubseteq \sup_{i\in I}a_i$. Since  $\dbt \in{\bf tot}$, then $a_i \in{\bf tot}$ for some $i\in I$ by compactness of $\mathcal{L}$. Since   $\dbt\in{\bf con}$ and $\dbt\sqcap\dbf=a_i\sqcap\dbf$, it follows from ({\bf con}-{\bf tot}) that $\dbt\sqsubseteq a_i$. Therefore  $\dbt$ is a finite element of  $(L_{+},\sqsubseteq)$. Likewise, $\dbf$ is a finite element of  $(L_{-},\sqsubseteq)$.
\end{proof} 

\begin{prop}\label{d-lattice--d-frame}
A d-frame $\mathcal{L}=(L;\dbt,\dbf;{\bf con},{\bf tot})$ is isomorphic to the d-frame of ideals of a d-lattice if and only if it is compact and the underlying frame $L$ is coherent. 
\end{prop}

\begin{proof} It is clear that for each d-lattice $\mathcal{M}=(M;\dbt,\dbf;{\bf con},{\bf tot})$, the d-frame $\idl\mathcal{M}$ is compact and the frame $\idl M$ is coherent, so the necessity follows.

For sufficiency, suppose $\mathcal{L}=(L;\dbt,\dbf;{\bf con},{\bf tot})$ is a compact d-frame with the underlying frame $L$ being coherent. Let $K(L)$ be the set of finite elements of $(L,\sqsubseteq)$. By Lemma \ref{t and f are comapct} both $\dbt$ and $\dbf$ belong to $K(L)$.  Since $L$ is a coherent frame, $K(L)$ is a sublattice of $L$, hence a distributive lattice. The  structure $$K(\mathcal{L})\coloneqq (K(L);\dbt,\dbf;{\bf con}_{K(L)},{\bf tot}_{K(L)})$$ is then a d-lattice, where   ${\bf con}_{K(L)}={\bf con}\cap K(L)$ and ${\bf tot}_{K(L)}={\bf tot}\cap K(L)$. We assert that  $\mathcal{L}$ is isomorphic to the d-frame of ideals of $K(\mathcal{L})$.
  
Define $\epsilon\colon \idl K(L)\longrightarrow L$ and $\kappa\colon L\longrightarrow\idl K(L)$  by
$$\epsilon(I)=\sup I \quad\text{and}\quad \kappa(a)=\thinspace \downarrow\! a\cap K(L). $$
Since $L$ is a coherent frame, both $\kappa$ and $\epsilon$ are frame homomorphisms and are inverse to each other.
Thus, what remains to check is that both $\kappa\colon \mathcal{L}\lra \idl K(\mathcal{L})$ and $\epsilon\colon  \idl K(\mathcal{L})\lra\mathcal{L} $ preserve consistency and totality. We check that $\kappa$ preserves totality for example. Suppose $a\in L_+$, $b\in L_-$ and $a\sqcup b\in{\bf tot}$. Since $\kappa(a)$  is a directed subset of $L_+$ and $\kappa(b)$ is a directed subset of $L_-$, it follows that  $\{x\sqcup y :x\in\kappa(a), y\in\kappa(b)\}$ is a directed subset of $L$. Since $a$ is the join of $\kappa(a)$, $b$ is the join of $\kappa(b)$, it follows that $a\sqcup b$ is the join of $ \{x\sqcup y:x\in\kappa(a),y\in\kappa(b)\}$, hence  
by compactness of $\mathcal{L}$, there exist some $x\in\kappa(a)$ and $y\in\kappa(b)$ such that $x\sqcup y\in{\bf tot}$. So  $(\kappa(a)\sqcup\kappa(b))\cap {\bf tot}\neq\emptyset$. Therefore, $\kappa(a\sqcup b)$, which is equal to $\kappa(a)\sqcup\kappa(b)$, belongs to the totality predicate of the d-frame $\idl K(\mathcal{L})$. \end{proof}

A d-frame  $\mathcal{L}=(L;\dbt,\dbf;{\bf con},{\bf tot})$ is called a \emph{coherent d-frame} if it is compact and the underlying frame $L$ is coherent. In other words, coherent d-frames are the d-frames of ideals of d-lattices. A d-frame homomorphism $f\colon \mathcal{L}\lra\mathcal{M}$ between coherent d-frames is  \emph{coherent} if the underlying frame homomorphism $f\colon L\lra M$   preserves finite elements. 

\begin{cor}\label{d-lattice vs d-frame} The category of d-lattices is equivalent to the category of  coherent d-frames and coherent d-frame homomorphisms. \end{cor}
\begin{proof} This follows from Proposition \ref{d-lattice--d-frame} immediately. \end{proof}

Now we introduce the notion of d-Boolean algebras. The relationship between d-Boolean algebras and  d-lattices is analogous to that between Boolean algebras and distributive lattices.

\begin{lem}\label{uniqueness of adag} Let $(L;\dbt,\dbf;{\bf con},{\bf tot})$ be a d-lattice. Then for each element $a$ of $L_+$, there is at most one element $b$ of $L_-$ for which $a\sqcup b\in {\bf con}\cap {\bf tot}$. Likewise, for each element $b$ of $L_-$, there is at most one element $a$ of $L_+$ for which $a\sqcup b\in {\bf con}\cap {\bf tot}$.   \end{lem} 

\begin{proof} Suppose   $b_1$ and $b_2$ are elements of $L_-$ such that both $a\sqcup b_1$ and $a\sqcup b_2$ belong to the intersection ${\bf con}\cap {\bf tot}$. Then $a\sqcup(b_1\sqcup b_2)\in{\bf con}$ and  $a\sqcup (b_1\sqcap b_2)\in{\bf tot}$. Since $$(a\sqcup(b_1\sqcup b_2))\sqcap \dbt=a=(a\sqcup (b_1\sqcap b_2))\sqcap  \dbt , $$ it follows from  ({\bf con}--{\bf tot}) that $a\sqcup(b_1\sqcup b_2)\sqsubseteq a\sqcup (b_1\sqcap b_2))$, hence  $$b_1\sqcup b_2= (a\sqcup(b_1\sqcup b_2))\sqcap\dbf\sqsubseteq (a\sqcup (b_1\sqcap b_2)))\sqcap\dbf = b_1\sqcap b_2,$$  then $b_1=b_2$, as desired. \end{proof}

\begin{prop}\label{d-complement is order-reversing} Suppose $(L;\dbt,\dbf;{\bf con},{\bf tot})$ is a d-lattice, $a_1,a_2\in L_+$, and   $b_1,b_2\in L_-$.   If $a_1\sqcup b_1, a_2\sqcup b_2 \in{\bf con} \cap {\bf tot}$, then $a_1\sqsubseteq a_2$ implies   $b_1\sqsupseteq b_2$. Therefore, ${\bf con} \cap {\bf tot}$ is an anti-chain of $L$ under the information order $\sqsubseteq$. \end{prop}

\begin{proof} Since ${\bf con} \cap {\bf tot}$ is a sublattice of $L$ with respect to the logic order $\leq$, it follows that   $(a_1\sqcup a_2)\sqcup(b_1\sqcap b_2)$, which is the join of $a_1\sqcup b_1$ and $a_2\sqcup b_2$ under the logic order $\leq$, belongs to ${\bf con} \cap {\bf tot}$.   If $a_1\sqsubseteq a_2$, then $a_1\sqcup a_2=a_2$, hence $b_2= b_1\sqcap b_2$ by Lemma \ref{uniqueness of adag}, which implies $b_1\sqsupseteq b_2$. \end{proof}

\begin{defn}(\cite[Definition 2.2.2]{Klinke})
Let $(L;\dbt,\dbf;{\bf con},{\bf tot})$ be a d-lattice and $x\in L$. We say that $x$ is  d-complemented if,  \begin{itemize} \setlength{\itemsep}{0pt} \item either $x\in L_+$ and there is  some   $x^\dag\in L_-$ such that $x\sqcup x^\dag\in{\bf con} \cap {\bf tot}$, \item or $x\in L_-$ and there is some  $x^\dag\in L_+$ such that $x^\dag\sqcup x\in{\bf con} \cap {\bf tot}$. \end{itemize}  In this case $x^\dag$ is called a  d-complement of $x$. \end{defn}

\begin{defn} \label{defn of d-boolean} A d-lattice $(L;\dbt,\dbf; {\bf con}, {\bf tot})$ is called a d-Boolean algebra if it is d-complemented in the sense that all elements of $L_+$ and all elements  of $L_-$ are d-complemented. 
The full subcategory of $\dLat$ composed of d-Boolean algebras is denoted by  $\dBoo.$ \end{defn}  

For each distributive lattice $L$, the d-lattice $\omega(L)$ is a d-Boolean algebra if and only if $L$ is a Boolean algebra. So, we have a full and faithful functor $$\omega\colon{\bf BA}\lra{\bf dBA}.$$   

Suppose $(L;\dbt,\dbf;{\bf con},{\bf tot})$ is a  d-Boolean algebra. By Lemma \ref{uniqueness of adag} and Proposition \ref{d-complement is order-reversing} one sees that taking d-complement defines an order-reversing isomorphism $ ^\dagger\colon L_+\lra L_-$. Furthermore, the consistency predicate {\bf con} and the totality predicate {\bf tot} are determined by the order-reversing isomorphism as follows:  $${\bf con}=\{a\sqcup b: a\in L_+, b\in  L_-, a^\dag\sqsupseteq b \},\quad  {\bf tot}=\{a\sqcup b: a\in L_+, b\in  L_-,a^\dag\sqsubseteq b\}.$$ This shows that the structure of a d-Boolean algebra is completely determined by the pair $(L_+, L_-)$ of distributive lattices together with the order-reversing isomorphism $^\dagger\colon L_+\lra L_-$.  

As shall be seen below, the category of d-Boolean algebras is equivalent to the category of distributive lattices, though d-Boolean algebras look a bit different from distributive lattices. To see this, we define a category {\bf DBL} as follows: \begin{itemize}\setlength{\itemsep}{0pt} \item objects: an object is a triple $(L_+,L_-, ^\dag)$, where $L_+,L_-$ are distributive lattices and $ ^\dag\colon L_+\lra L_-$ is an order-reversing isomorphism. \item morphisms: a morphism from $(L_+,L_-, ^\dag)$ to $(M_+,M_-, ^\dag)$ is a pair $(f,g)$, where $f\colon L_+\lra M_+$ and $g\colon L_-\lra M_-$ are two lattice homomorphisms such that $g(a^\dag)=f(a)^\dag$ for all $a\in L_+$. $$\bfig\square[L_+` L_-`M_+` M_-; ^\dag`f`g` ^\dag] \efig $$\end{itemize}

\begin{lem}\label{d-Bool as DBL} The category $\dBoo$ of d-Boolean algebras is isomorphic to the category {\bf DBL}. \end{lem}

\begin{proof} For each d-Boolean algebra $(L;\dbt,\dbf;{\bf con},{\bf tot})$, the triple $(L_+,L_-, ^\dag)$ is an object of {\bf DBL}, where $ ^\dag\colon L_+\lra L_-$ takes each $a\in L_+$ to its d-complement, which is an element of $L_-$. Suppose $h\colon (L;\dbt,\dbf;{\bf con},{\bf tot})\lra (M;\dbt,\dbf;{\bf con},{\bf tot})$ is a morphism (i.e. a d-lattice homomorphism) between d-Boolean algebras. Since $h$ maps $L_+$ to $M_+$ and maps $L_-$ to $M_-$, both the restriction $f$ of $h$ to $L_+$ and $M_+$ and the restriction $g$ of $h$ to $L_-$ and $M_-$  are lattice homomorphisms. Since $h$ preserves {\bf con} and {\bf tot}, it follows that for each $a\in L_+$, $h(a^\dag)$ is the d-complement of $h(a)$; that is, $g(a^\dag)=f(a)^\dag$. So we obtain a functor  $F\colon\dBoo\lra {\bf DBL}.$  

For each object $(L_+,L_-, ^\dag)$ of {\bf DBL}, it is readily seen that $(L_+\times L_-;\dbt,\dbf;{\bf con},{\bf tot})$ is a d-Boolean algebra, where \begin{itemize}\setlength{\itemsep}{0pt}
\item   $\dbt=(1,0)$, $\dbf=(0,1)$, \item ${\bf con}=\{(a,b)\in L_+\times L_-: a^\dag\sqsupseteq b\}$, \item ${\bf tot}=\{(a,b)\in L_+\times L_-: a^\dag\sqsubseteq b\}$.\end{itemize} For each morphism $(f,g)\colon (L_+,L_-, ^\dag)\lra (M_+,M_-, ^\dag)$ in the category {\bf DBL},  the product $f\times g$  is a morphism $(L_+\times L_-;\dbt,\dbf;{\bf con},{\bf tot})\lra (M_+\times M_-;\dbt,\dbf;{\bf con},{\bf tot})$ between  d-Boolean algebras. So we obtain a functor  $G\colon {\bf DBL}\lra\dBoo.$   

The functors $F$ and $G$ are readily verified to be inverse to each other, so the categories $\dBoo$ and {\bf DBL} are isomorphic to each other. \end{proof}

\begin{cor}\label{dBA via order reversing} The consistency predicate {\bf con} of each d-Boolean algebra is a Scott closed set under the information order. \end{cor} 

\begin{proof} It suffices to check that for each object $(L_+,L_-, ^\dag)$ of the category {\bf DBL},   the consistency predicate of the corresponding  d-Boolean algebra $(L_+\times L_-;\dbt,\dbf;{\bf con},{\bf tot})$ is closed under directed joins; that is, the set  $${\bf con}=\{(a,b)\in L_+\times L_-: a^\dag\sqsupseteq b\}$$  is closed under directed joins in the product lattice $L_+\times L_-$. 

Suppose  $\{(a_i,b_i)\}_{i\in D}$ is a directed subset of ${\bf con}$ with a join $(a,b)$ in $L_+\times L_-$. First we show that each $a_j^\dag$ is an upper bound of $\{b_i\}_{i\in D}$, hence $b\sqsubseteq a_j^\dag$. For each $i\in D$, pick some $k\in D$ such that $(a_i,b_i) \sqsubseteq (a_k,b_k)$ and $(a_j,b_j)\sqsubseteq (a_k,b_k)$. Then  $b_i\sqsubseteq b_k\sqsubseteq  a_k^\dag \sqsubseteq  a_j^\dag, $  which shows that $ a_j^\dag$ is an upper bound of $\{b_i\}_{i\in D}$ in $L_-$. Next we show that $(a,b)\in{\bf con}$. Since the map $ ^\dag\colon L_+\lra L_-$ transforms joins to meets, it follows that $$ a^\dag=\big(\sup_{i\in D} a_i\big)^\dag= \inf_{i\in D}a_i^\dag \sqsupseteq b,$$ and consequently, $(a,b)\in{\bf con}$. \end{proof}

\begin{prop}\label{DLat and DBL} The category of d-Boolean algebras is  equivalent to the category of distributive lattices. 
\end{prop}

\begin{proof} 
Let $L$ be a distributive lattice. Since the correspondence  $x\mapsto x$ is an order-reversing isomorphism $L\lra L^{\rm op}$,  it follows that $$\lambda(L)\coloneqq (L\times L^{\rm op} ;\dbt,\dbf;{\bf con},{\bf tot})$$ is a d-Boolean algebra, where  \begin{itemize} \setlength{\itemsep}{0pt} 
\item $\dbt=(1,1)$, $\dbf=(0,0)$; \item ${\bf con}=\{(a,b)\in L\times L: a\sqsubseteq b\}$;   \item ${\bf tot}=\{(a,b)\in L\times L: a\sqsupseteq b\}$. \end{itemize} \begin{center}
\begin{tikzpicture}[scale=0.9] \fill[lightgray] (0,0) -- (3,0) -- (3,3) -- cycle;
\draw[thick] (0,0) -- (3,0) -- (3,3) -- cycle;
\draw[thick] (0,0) -- (0,3) -- (3,3) -- cycle;
\node at (2,1) {{\bf tot}}; \node at (1,2) {{\bf con}}; 
\node at (3.3,3) {$\dbt$}; \node at (-0.4,0) {$\dbf$};
\end{tikzpicture} \end{center}  

In this way we obtain a functor  $$\lambda\colon{\bf DisLat}\lra\dBoo,$$ which is an equivalence of categories, as we shall see.  

Since the consistency predicate and the totality predicate  of the d-Boolean algebra $\lambda(L)$ are  determined by the order relation of $M$, it is clear that the functor $\lambda$ is full and faithful. It remains to check that it is essentially surjective on objects.  For this we show that each d-Boolean algebra $(L;\dbt,\dbf;{\bf con},{\bf tot})$ is isomorphic to $\lambda(L_+)$.

Since the square  $$ \bfig \square[L_+`L_-`L_+`L_+^{\rm op}; ^\dag`{\rm id}`^\dag`{\rm id}] \efig$$ is commutative,  where $ ^\dag$ is the map sending each element to its d-complement,   {\rm id} is the identity map on the set $L_+$, it follows that $(L_+,L_-,^\dag)$ is isomorphic to $(L_+,L_+^{\rm op},{\rm id})$ in the category {\bf DBL}, then the conclusion follows from Lemma \ref{d-Bool as DBL}. \end{proof} 
  
The next proposition says that the category of d-Boolean algebras is a coreflective  subcategory of that of d-lattices. 
For each d-lattice $(L;\dbt,\dbf;{\bf con},{\bf tot})$, let $$B_+=\{a\in L_+: a~\text{is d-complemented}\}, \quad B_-=\{b\in L_-: b~\text{is d-complemented}\}.$$ Then  $B_+$ is a sublattice of $L_+$;   $B_-$ is a sublattice  of $L_-$.  
Let $$\dB L=\{a\sqcup b: a\in B_+,b\in B_-\}. $$ Then $\dB L$ is a sublattice of $L$ and contains $\dbt$ and $\dbf$; the structure $$\dB\mathcal{L}\coloneqq(\dB L;\dbt,\dbf;{\bf con}_{\dB},{\bf tot}_{\dB})$$ is a d-Boolean algebra,  where $${\bf con}_{\dB}={\bf con}\cap \dB L, \quad {\bf tot}_{\dB}={\bf tot}\cap \dB L.$$ 
Thus, we have a functor $$\dB\colon \dLat \longrightarrow \dBoo.$$
The following proposition says that $\dB\mathcal{L}$ is the d-Boolean algebra coreflection of  $\mathcal{L}$.   

\begin{prop} The functor $\dB\colon \dLat \longrightarrow \dBoo$ is right adjoint to the inclusion functor $V\colon \dBoo\longrightarrow\dLat$. 
\end{prop}

\begin{proof} We show that for each  d-lattice  $\mathcal{L}=(L;\dbt,\dbf;{\bf con},{\bf tot})$, the d-Boolean algebra  $\dB\mathcal{L}=(\dB L;\dbt,\dbf;{\bf con}_{\dB},{\bf tot}_{\dB})$ is its $\dBoo$-coreflection. The inclusion map $i\colon \dB L\lra L$ is clearly a d-lattice homomorphism $\dB\mathcal{L}\lra\mathcal{L}$, so, it suffices to check that each d-lattice homomorphism $f\colon \mathcal{M}\lra\mathcal{L}$ with $\mathcal{M}= (M;\dbt,\dbf;{\bf con},{\bf tot})$ being a d-Boolean algebra factors through  $i\colon\dB\mathcal{L}\lra\mathcal{L}$. $$\bfig\ptriangle/->`<--`<-/[\dB\mathcal{L}`\mathcal{L}`\mathcal{M};i` `f]\efig$$
This follows directly from the fact that  
if $x\in M$ is d-complemented then so is $f(x)$. \end{proof}

\begin{exmp}
Suppose $(X, \tau_+, \tau_-)$ is a bitopological space. Let  $$L_+=\{U\in\tau_+: X\setminus U\in\tau_-\},\quad L_-=\{V\in\tau_-: X\setminus V\in\tau_+\}.$$ Put differently, an element of $L_+$ is a $\tau_+$-open and $\tau_-$-closed set; likewise for $L_-$. Since for each subset $U$ of $X$, $ U\in L_+$ if and only if $X\setminus U\in L_-$,   the correspondence $U\mapsto X\setminus U$ is an order-reversing isomorphism   $L_+\lra L_-$. So,  the structure $(L;\dbt,\dbf;{\bf con},{\bf tot})$  is a d-Boolean algebra, where  \begin{itemize}\setlength{\itemsep}{0pt}
\item $L=L_+\times L_-$; \item $\dbt=(X,\emptyset)$, $\dbf=(\emptyset,X)$; \item ${\bf con}=\{(U,V)\in L_+\times L_-: U\cap V=\emptyset\}$;   \item ${\bf tot}=\{(U,V)\in L_+\times L_-: U\cup V=X\}$. \end{itemize}  
The d-Boolean algebra $(L;\dbt,\dbf;{\bf con},{\bf tot})$  is called the \emph{d-Boolean algebra of d-clopen sets} of $(X, \tau_+, \tau_-)$. In this way we obtain a contravariant functor  $$\dClop\colon{\bf BiTop}\longrightarrow{\bf dBA}^{\rm op}.$$   

Since a $\tau_+$-open set $U$ is d-complemented in the d-frame  $\dO(X, \tau_+, \tau_-)$   if and only if $U$ is $\tau_-$-closed, a $\tau_-$-open set $V$ is d-complemented in   $\dO(X, \tau_+, \tau_-)$  if and only if $V$ is $\tau_+$-closed, it follows that the d-Boolean algebra of d-clopen sets of $(X,\tau_+,\tau_-)$ is the d-Boolean algebra coreflection of the d-frame $\dO (X,\tau_+,\tau_-)$. That means, $\dClop=\dB\circ\dO.$    \end{exmp}

Composing  $V\dashv\dB\colon \dLat\lra\dBoo$ with   $\idl\dashv U\colon \dFrm\lra\dLat$ yields an adjunction $$\idl\dashv \dB\colon \dFrm\lra \dBoo.$$ The ``fixed points'' of this adjunction presents a representation of d-Boolean algebras by d-frames. 
A d-frame  $\mathcal{L}=(L;\dbt,\dbf;{\bf con},{\bf tot})$ is said to be \emph{d-zero-dimensional}  \cite[Definition 2.3.6]{Jakl}  if every element of $L$ is the join of a set of d-complemented elements.  
  
\begin{thm} \label{dBA vs dfrm}
The category of d-Boolean algebras is equivalent to the category of compact and d-zero-dimensional d-frames.
\end{thm}

\begin{proof} 
Write {\bf KZdFrm} for the category of compact and d-zero-dimensional d-frames. Then by Proposition \ref{one side of the duality} and Proposition \ref{other side of the duality} below, the functors $\dB\colon {\bf KZdFrm}\lra\dBoo$ and $\idl\colon \dBoo\lra{\bf KZdFrm}$ witness the equivalence of the categories $\dBoo$ and {\bf KZdFrm}. \end{proof}

\begin{prop}\label{one side of the duality} Each   d-Boolean algebra $\mathcal{L} =(L;\dbt,\dbf; {\bf con},{\bf tot})$  is isomorphic to the d-Boolean algebra    $\dB\circ\idl \mathcal{L}$. 
\end{prop}

\begin{proof} This follows immediately from the following lemma.
\end{proof}

\begin{lem}\label{d-complemented ideal} Let  
$\mathcal{L} =(L;\dbt,\dbf; {\bf con},{\bf tot})$ be a d-lattice. Then an ideal $I$ of the distributive lattice $L$ is d-complemented in the d-frame   $\idl\mathcal{L}$ if and only if $I= \thinspace \downarrow\!x$ for some d-complemented element $x$ of $\mathcal{L}$. \end{lem} 

\begin{proof} For sufficiency suppose $x$ is a d-complemented element of $\mathcal{L}$. Then by definition, either $x\in L_+$ or $x\in L_-$. Without loss of generality we assume that  $x\in L_+$.  Then $\downarrow\!x\sqsubseteq\thinspace \downarrow\!\dbt$ and $\downarrow\!x^\dag\sqsubseteq\thinspace \downarrow\!\dbf$. In the frame $\idl L$ it holds that $(\downarrow\!x)\sqcup(\downarrow\!x^\dag)= \thinspace \downarrow\!(x\sqcup x^\dag)$, the latter clearly belongs to ${\bf con}_{\idl}\cap{\bf tot}_{\idl},$  so the ideal $\downarrow\!x^\dag$ is a d-complement of the ideal $\downarrow\!x$ in the d-frame $\idl\mathcal{L}$, which implies that $\downarrow\!x$ is  d-complemented. 

For necessity  suppose $I$ is a d-complemented element of $\idl \mathcal{L}$. By definition either $I\subseteq L_+$ or $I\subseteq L_-$. Without loss of generality we assume that $I\subseteq L_+$. Since $I$ is d-complemented, there is an ideal $J$ of $L$ contained in $L_-$ for which the join $I\sqcup J$ belongs to both ${\bf con}_{\idl}$ and ${\bf tot}_{\idl}$. So there exist $a\in I$ and $b\in J$ such that $a\sqcup b\in{\bf con}\cap{\bf tot}$. We finish the proof by  showing that $I=\thinspace\downarrow\!a$. For each $a'\in I$ take an upper bound $a''$ of $a$ and $a'$ in $I$, then $a''\sqcup b\in I\sqcup J$, which implies $a''\sqcup b\in {\bf con}\cap{\bf tot}$, hence $a=a''$ by Lemma \ref{uniqueness of adag} and consequently, $a'\sqsubseteq a$ as desired.  \end{proof}

\begin{prop}\label{other side of the duality} A d-frame  $\mathcal{L}=(L;\dbt,\dbf;{\bf con},{\bf tot})$ is   compact and d-zero-dimensional if and only if $\mathcal{L}$ is isomorphic to the d-frame   $ \idl\circ\dB \mathcal{L}$. 
\end{prop}
 
\begin{lem}\label{each element of dBL is finite}
Let $\mathcal{L}=(L;\dbt,\dbf;{\bf con},{\bf tot})$ be a compact d-frame. Then every element of the underlying lattice $\dB L$  of its d-Boolean algebra coreflection $\dB\mathcal{L}$ is finite.  \end{lem}  

\begin{proof}
Since $\dB L$ is isomorphic to the product lattice $B_+\times B_-$, where $B_+$ and $B_-$ are, respectively, the sublattices  of d-complemented elements in $L_+$ and $L_-$, it suffices to show that every element of $B_+$ is finite, and that every element of $B_-$ is finite. 

Let $a\in B_+$  and let $D$ be a directed subset of $B_+$ such that $a\sqsubseteq\sup D$. Then $\{d\sqcup a^\dag:d\in D\}$ is a directed set of $L$ having $(\sup D)\sqcup a^\dag$ as a join. Since $a\sqcup a^\dag \in{\bf tot}$, then $(\sup D)\sqcup a^\dag\in{\bf tot}$, hence $d\sqcup a^\dag\in{\bf tot}$ for some $d\in D$ because {\bf tot} is Scott open. Since $a\sqcup a^\dag\in{\bf con}$, then $a\sqsubseteq d$ by ({\bf con}-{\bf tot}). Likewise, every element of $B_-$ is finite.  \end{proof}

\begin{proof}[Proof of Proposition \ref{other side of the duality}]  The d-frame of ideals of each d-Boolean algebra is readily verified to be compact and d-zero-dimensional, the sufficiency thus follows. For necessity, we show that if  $\mathcal{L}=(L;\dbt,\dbf;{\bf con},{\bf tot})$ is a compact and d-zero-dimensional d-frame, then it is isomorphic to  the d-frame $ \idl\circ\dB \mathcal{L}$. Define $\epsilon\colon\idl(\dB L)\longrightarrow L$ and $\kappa\colon L\lra \idl(\dB L)$ by $$\epsilon(I)=\sup I  \quad \text{and}\quad \kappa(x) = \thinspace\downarrow\!x  \cap\dB L  $$ for each ideal $I$ of $\dB L$ and each element $x$ of $L$. 

Both   $\epsilon\colon \idl\circ\dB  \mathcal{L}  \longrightarrow \mathcal{L} $ and $\kappa\colon  \mathcal{L}\lra \idl\circ\dB  \mathcal{L} $ are d-frame homomorphisms. The verification is routine,  we  check that $\kappa$ preserves the totality predicate  for example. Suppose $x\in{\bf tot}$. Since $\mathcal{L}$ is d-zero-dimensional, the join of $\kappa(x)$ is $x$.  Since {\bf tot} is Scott open, then $\kappa(x)$ meets {\bf tot}, hence $\kappa(x)$ belongs to the totality predicate of the d-frame $ \idl\circ\dB \mathcal{L}$. 

Since $\mathcal{L}$ is d-zero-dimensional,  the composite $\epsilon\circ\kappa$ is the identity on $L$. Since $\mathcal{L}$ is compact and d-zero-dimensional, Lemma \ref{each element of dBL is finite} ensures that the composite $\kappa\circ\epsilon$ is the identity on $\idl(\dB L)$. Therefore, $\mathcal{L}$ is isomorphic to  $\idl\circ\dB \mathcal{L}$. 
\end{proof}

\begin{rem}   
It is proved in Jakl \cite[Section  2.6]{Jakl} that the category of distributive lattices is equivalent to the category of compact and d-zero-dimensional d-frames. 
The equivalence in Theorem \ref{dBA vs dfrm} is an immediate consequence of the result of Jakl and the equivalence between distributive lattices and d-Boolean algebras (Proposition \ref{DLat and DBL}). As pointed out by an anonymous referee, the equivalence   $\mathcal{IF}$, constructed in \cite[Section  2.6]{Jakl}, between the category of distributive lattices and the category of compact and d-zero-dimensional d-frames coincides with the composite   $\idl\circ\lambda\colon {\bf DisLat} \lra \dBoo \lra{\bf KZdFrm}.$   \end{rem}

Let {\bf KZFrm} be the full subcategory of the category {\bf Frm}  composed of compact and d-zero-dimensional frames; let $\Ba\colon {\bf Frm}\lra{\bf BA}$ be the functor that sends each frame to the Boolean algebra of its complemented elements. Restricting the  domain of the functor $\Ba\colon {\bf Frm}\lra{\bf BA}$ to {\bf KZFrm} gives rise to an equivalence between the category of Boolean algebras and the category of compact and zero-dimensional frames. This is actually the frame version (or, point-free version) of the Stone representation of Boolean algebras. 
 
We end this section with a representation of complete d-Boolean algebras by d-frames. A d-Boolean algebra $(L;\dbt,\dbf;{\bf con}, {\bf tot})$ is said to be \emph{complete} if the distributive lattice $L$ is complete. It is trivial that a d-Boolean algebra $(L;\dbt,\dbf;{\bf con}, {\bf tot})$ is complete if and only if both $L_+$ and $L_-$ are complete lattices.  We state the conclusion first.

\begin{prop} \label{complete dBA vs dfrm} The  category of complete d-Boolean algebras and d-Boolean algebra homomorphisms is equivalent to the category of extremally disconnected, compact and d-zero-dimensional d-frames.
\end{prop}

Recall that an element $a$ of a  lattice $L$ is \emph{pseudo-complemented} if there is some $\neg a$ of $L$ (necessarily unique) such that for all $b\in L$, $a\sqcap b=0\iff b\sqsubseteq \neg a$. For example, every element of a frame is pseudo-complemented.

\begin{defn}
Let $(L;\dbt,\dbf;{\bf con},{\bf tot})$ be a d-lattice and $x\in L$. We say that $x$ is  d-pseudo-complemented if, \begin{itemize} \setlength{\itemsep}{0pt} \item either $x\in L_+$ and there is  some  $x^*\in L_-$  such that   $x\sqcup b\in{\bf con}\iff b\sqsubseteq x^*$ for all $b\in L_-$, \item or $x\in L_-$ and there is some $x^*\in L_+$ such that  $a\sqcup x\in{\bf con}\iff a\sqsubseteq x^*$ for all $a\in L_+$. \end{itemize} In this case $x^*$ is called a  d-pseudo-complement of $x$. \end{defn}

We say that a d-lattice  $(L;\dbt,\dbf; {\bf con}, {\bf tot})$ is  \emph{d-pseudo-complemented} if all elements of $L_+$ and $L_-$ are d-pseudo-complemented. In this case, for all $a\in L_+$ and $b\in L_-$  it holds that $$b\sqsubseteq a^*\iff a\sqcup b\in{\bf con}\iff a\sqsubseteq b^*.$$

For each distributive lattice $L$, the d-lattice $\omega(L)$ is d-pseudo-complemented if and only if $L$ is pseudo-complemented. Every d-frame  is d-pseudo-complemented \cite{MMW}. 
The  following proposition relates d-complement to d-pseudo-complement.

\begin{prop}\label{d-complement as d-pseudo-complement} An element $x$ of a d-lattice $(L;\dbt,\dbf;{\bf con},{\bf tot})$   is d-complemented if and only if $x$  is d-pseudo-complemented and $x\sqcup x^*\in {\bf tot}$.   In this case  $x^*=x^\dag$.   
\end{prop}

\begin{defn}\label{extremally disconnected d-frame} A d-frame $\mathcal{L}=(L;\dbt,\dbf;{\bf con},{\bf tot})$ is extremally disconnected if the d-pseudo-complement of each element  of $L_+\cup L_-$ is d-complemented. \end{defn}

For each frame $L$,   the d-frame $\omega(L)$ is extremally disconnected if and only if the frame $L$ is extremally disconnected  \cite[page 101]{Johnstone} in the sense that $\neg a\sqcup\neg\neg a=1$ for every $a$ in $L$.  

\begin{lem}Let $\mathcal{L}=(L;\dbt,\dbf;{\bf con},{\bf tot})$ be a d-frame. Then the following are equivalent: \begin{enumerate}[label=\rm(\arabic*)] \setlength{\itemsep}{0pt}
\item $\mathcal{L}$ is extremally disconnected.
\item For each element $x$ of $L_+$, $x^{**}\sqcup x^{*}\in{\bf tot}$; and likewise for elements of $L_-$.
\item For all $a\in L_{+}$ and $b\in L_{-}$,  $a\sqcup b\in{\bf con}\iff b^*\sqcup a^*\in{\bf tot}$. \end{enumerate} \end{lem}

\begin{proof}
$(1)\Rightarrow(2)$ Let $x\in L_{+}$. Since $\mathcal{L}$ is extremally disconnected,  $x^*$ is d-complemented.  By Proposition \ref{d-complement as d-pseudo-complement} the d-complement of $x^*$ is given by its d-pseudo-complement $x^{**}$, so $x^{**}\sqcup x^{*}\in{\bf tot}$. Likewise for elements of $L_-$.

$(2)\Rightarrow(3)$ Let $a\in L_{+}$ and $b\in L_{-}$. If $a\sqcup b\in{\bf con}$, then $b\sqsubseteq a^*$ by definition, hence $a^{**}\sqsubseteq b^*$. Since $a^{**}\sqcup a^{*}\in{\bf tot}$ by assumption, it follows that $b^*\sqcup a^*\in{\bf tot}$ because {\bf tot} is an upper set  under the information order. This shows that $a\sqcup b\in{\bf con}\implies b^*\sqcup a^*\in{\bf tot}$. The converse implication holds for all d-frames. If $b^*\sqcup a^*\in{\bf tot}$, then from $a\sqcup a^*\in{\bf con}$ and ({\bf con}--{\bf tot}) it follows that $a\sqsubseteq b^*$, hence $a\sqcup b\in{\bf con}$.  

$(3)\Rightarrow(1)$ We show that the d-pseudo-complement $a^*$ of each  $a\in L_{+}$ is d-complemented. Since $a\sqcup a^*\in{\bf con}$, then $a^{**}\sqcup a^*\in{\bf tot}$ by assumption, hence $a^*$ is d-complemented by Proposition \ref{d-complement as d-pseudo-complement}. Likewise, the d-pseudo-complement $b^*$ of each  $b\in L_{-}$ is d-complemented. Therefore, $\mathcal{L}$ is extremally disconnected.
\end{proof}

The following lemma shows that for a d-zero-dimensional d-frame, extremal disconnectedness of $\mathcal{L}$ is equivalent to completeness of its d-Boolean algebra coreflection. This extends the characterization (v) for frames on page 101 of Johnstone \cite{Johnstone} to the realm of d-frames.

\begin{lem} \label{complete dBA vs extremal discon} Let $\mathcal{L}=(L;\dbt,\dbf;{\bf con},{\bf tot})$ be  d-frame.   If  $\mathcal{L}$  is extremally disconnected, then the d-Boolean algebra $\dB\mathcal{L}$ is complete.
The converse implication also holds provided that   $\mathcal{L}$ is d-zero-dimensional.  
\end{lem}

\begin{proof} Suppose $\mathcal{L}$ is extremally disconnected. To see that the d-Boolean algebra $\dB\mathcal{L}$ (i.e., the d-Boolean coreflection of $\mathcal{L}$) is complete, it suffices to show that both $B_+$ and $B_-$ are complete lattices. Since taking d-pseudo-complement defines an anti-tone Galois connection between the complete lattices $L_+$ and $L_-$, the subset $\{a^*:a\in L_+\}$  is a complete lattice under the information order. By extremal disconnectedness of $\mathcal{L}$, $B_-$ is equal to the set $\{a^*:a\in L_+\}$, so $B_-$ is a complete lattice. Likewise, $B_+$ is a complete lattice. 

Now we show that if $\mathcal{L}$ is d-zero-dimensional and the d-Boolean algebra $\dB\mathcal{L}$ is complete, then $\mathcal{L}$ is extremally disconnected. Here we only check that the d-pseudo-complement  of each element of $L_+$ is d-complemented. The case for elements of $L_-$ is similar. 

Let $a\in L_+$. Since every element of $B_-$ is d-complemented, it suffices to show that the d-pseudo-complement $a^*$ of $a$ belongs to $B_-$. To this end, we show that  $$a^*={\sup}_{B_-} \{b\in B_{-}:b\sqsubseteq a^*\},$$ where the symbol ${\sup}_{B_-}$ denotes  join  in the complete lattice $B_-$. We proceed with two steps. 

{\bf Step 1}.   $a^*\sqsubseteq{\sup}_{B_-} \{b\in B_{-}:b\sqsubseteq a^*\}$. 

Since $\mathcal{L}$ is d-zero-dimensional, $a^*$ is a join (in $L$) of d-complemented elements (which are necessarily in $B_-$), then $a^*=\sup_L \{b\in B_{-}:b\sqsubseteq a^*\}$, hence $a^*\sqsubseteq{\sup}_{B_-} \{b\in B_{-}:b\sqsubseteq a^*\}.$   

{\bf Step 2}.   $a^*\sqsupseteq{\sup}_{B_-} \{b\in B_{-}:b\sqsubseteq a^*\}$. 

It suffices to show that  $a\sqcup \big({\sup}_{B_-}\{b\in B_-:b\sqsubseteq a^*\}\big)\in{\bf con}.$ 
For each $x\in B_{+}$ with $x\sqsubseteq a$ and each $b\in B_{-}$ with $b\sqsubseteq a^*$, since $a\sqcup a^*\in{\bf con}$, then $x\sqcup b\in{\bf con}$, hence $x\sqcup b\in{\bf con}_{\dB}$. Since  $\dB\mathcal{L}$ is a d-Boolean algebra, the consistency predicate ${\bf con}_{\dB}$ is Scott closed in $(\dB L,\sqsubseteq)$ by Corollary \ref{dBA via order reversing}.  Since $\{b\in B_-:b\sqsubseteq a^*\}$ is directed under the information order $\sqsubseteq$, then $$x\sqcup\big({\sup}_{B_-}\{b\in B_-:b\sqsubseteq a^*\}\big)\in {\bf con}_{\dB}\subseteq{\bf con}$$ for all $x\in B_{+}$ with $x\sqsubseteq a$.   

Since $\mathcal{L}$ is d-zero-dimensional,   $a$ is equal to the join of $\{x\in B_{+}:x\sqsubseteq a\}.$   Since the set $\{x\in B_+:x\sqsubseteq a\}$ is  directed  and ${\bf con}$ is a Scott closed set of $(L,\sqsubseteq)$,   it follows that $$a\sqcup\big({\sup}_{B_-}\{b\in B_-:b\sqsubseteq a^*\}\big)\in{\bf con}, $$ as desired.  \end{proof}

\begin{proof}[Proof of Proposition \ref{complete dBA vs dfrm}] 
An immediate consequence of Theorem \ref{dBA vs dfrm} and Lemma \ref{complete dBA vs extremal discon}. \end{proof}

\section{Spectra of d-lattices and d-Boolean algebras} \label{d-spectrum}

This section concerns bitopological representation of d-Boolean algebras and d-lattices. Theorem \ref{dBA vs dfrm} represents d-Boolean algebras by compact and d-zero-dimensional d-frames;   Corollary \ref{d-lattice vs d-frame} represents d-lattices by coherent d-frames. If every coherent d-frame is spatial (for definition see below), then by the duality of Jung and Moshier \cite{JM2006,JM2008} between spatial d-frames and d-sober bitopological spaces  we would obtain bitopological representations for d-Boolean algebras and d-lattices. So, the problem reduces to whether every coherent d-frame is spatial; or equivalently, whether the d-frame of ideals of every d-lattice is spatial.

For each d-lattice $\mathcal{L}$, we introduce a bitopological space $\dSpec\mathcal{L}$, called the spectrum of $\mathcal{L}$. This space is proved to coincide with the space of d-points of the d-frame $\idl\mathcal{L}$. So, spectra of d-lattices are helpful in understanding the structure of bitopological spaces of d-points of coherent d-frames.  
It is proved that the spectrum of each  d-Boolean algebra is a d-Boolean space and every d-Boolean space arises in this way, so, the category of d-Boolean algebras is dually equivalent to the category of d-Boolean spaces. But, in contrast to the well-known fact that the frame of ideals of a distributive lattice is always spatial, the d-frame of ideals of a d-lattice need not be spatial, see Example \ref{idl is not always d-spatial}. 
  
We recall the duality of Jung and Moshier between spatial d-frames and d-sober bitopological spaces first. Let $\mathcal{L}=(L;\dbt,\dbf; \mathbf{con}, \mathbf{tot})$ be  a d-frame.  A \emph{d-point} \cite[Definition 3.4]{JM2008} of  $\mathcal{L}$  is a d-frame homomorphism $p\colon\mathcal{L}\lra\bbB$. The set $\dpt \mathcal{L}$ of d-points of  $\mathcal{L}$ becomes a bitopological space by considering  as the topology $\tau_+$ the collection of $$\Phi_+(a) =\{p :p(a)= \dbt\},\quad a\in L_+;$$   and as the topology $\tau_-$ the collection of $$\Phi_-(b) =\{p:p(b)= \dbf\},\quad b\in L_-.$$    
The construction for objects is extended to a  functor 
$$\dpt\colon \dFrm^{\rm op}\lra \bf{BiTop}$$ in the usual way, see \cite{JM2006,JM2008}.

\begin{thm}{\rm(\cite[Theorem 3.5]{JM2008})} The functor 
$\dpt \colon \dFrm^{\rm op}\lra \bf{BiTop}$ is right adjoint to  $\dO\colon\bf{BiTop}\lra \dFrm^{\rm op}$. 
\end{thm}

For later use we write out the unit and counit of the adjunction $\dO\dashv\dpt$. Suppose that $(X,\tau_+,\tau_-)$ is a bitopological space. For each $x\in X$, the map $$[x]\colon\tau_+\times\tau_-\lra\bbB, \quad [x](U,V) =\begin{cases}1 &x\in U\cap V, \\ \dbt &x\in U\setminus V,\\ \dbf & x\in V\setminus U,\\ 0 &x\notin U\cup V \end{cases}$$ is clearly a d-point of  $\dO(X,\tau_+,\tau_-)$. The component  $ (X,\tau_+,\tau_-)\lra \dpt\circ\dO(X,\tau_+,\tau_-)$  of unit of the adjunction $\dO\dashv\dpt$ at $(X,\tau_+,\tau_-)$ sends each $x$ of $X$ to the d-point $[x]$. To write out the counit, we identify the underlying frame $L$ of a d-frame $\mathcal{L}=(L;\dbt,\dbf; \mathbf{con}, \mathbf{tot})$ with the product frame $L_+\times L_-$. Then the  component of the counit at $\mathcal{L}$ sends each   $(a,b)$ of $L_+\times L_-$ to  $(\Phi_+(a),\Phi_-(b))$ of $\dO\circ\dpt\mathcal{L}$.
 
A bitopological space $(X,\tau_+,\tau_-)$ is said to be \emph{d-sober} if each d-point $p$ of  $\dO(X,\tau_+,\tau_-) $  is generated by a unique point $x$ of $X$ in the sense that $p=[x]$. As in the situation for topological spaces, a bitopological space $X$ is d-sober if and only if the component of the unit of the adjunction $\dO\dashv\dpt$ at $X$ is a bijection, hence a homeomorphism. It is known that every order-separated bitopological space is d-sober \cite[Theorem 4.13]{JM2006} and  every d-sober bitopological space is $T_0$.
A d-frame $\mathcal{L}$ is  \emph{spatial}  if it is the d-frame of open sets of some bitopological space; or equivalently, the component $\mathcal{L}\lra \dO(\dpt \mathcal{L})$ of the counit of the adjunction $\dO\dashv\dpt$ at $\mathcal{L}$ is an isomorphism.  The adjunction $\dO\dashv\dpt$ between $\bf{BiTop}$ and $\bf{dFrm}$ cuts down to a duality between  d-sober bitopological spaces and  spatial d-frames. See Jung and Moshier \cite{JM2006, JM2008} for details.

\begin{exmp} \label{omega(L) is spatial} For each frame $L$, the d-frame $\omega(L)$ is spatial if and only if $L$, as a frame, is spatial. In particular, for each distributive lattice $M$, the d-frame $\idl\omega(M)$ is spatial, because  $\idl M$ is a spatial frame and  $\idl\omega(M)$ is easily verified to be isomorphic to $\omega(\idl M)$. \end{exmp}

Before proceeding, some words on notations.  Suppose $\mathcal{L}=(L;\dbt,\dbf;{\bf con},{\bf tot})$ is a d-lattice and $I$ is an ideal of the distributive lattice $L$. Then $I_+=I\cap L_+$ is an ideal of the lattice $L_+$, $I_-=I\cap L_-$ is an ideal of the lattice $L_-$, and $I=\{a\sqcup b:a\in I_+,b\in I_-\}$. Conversely, for each ideal $I_+$ of $L_+$ and each ideal $I_-$ of $L_-$, the set $\{a\sqcup b:a\in I_+,b\in I_-\}$ is an ideal of $L$. These correspondences are clearly bijective,  so we often write an ideal of $L$ as a pair $(I_+,I_-)$. Written in this way, the consistency predicate and the totality predicate of  $\idl\mathcal{L}$ are given by $${\bf con}_{\idl}=\{(I_+,I_-): \forall a\in I_+,\forall b\in I_-, a\sqcup b \in {\bf con}\};$$
$${\bf tot}_{\idl}=\{(I_+,I_-): \exists a\in I_+,\exists b\in I_-, a\sqcup b \in {\bf tot}\}.$$

Let $L$ be a distributive lattice and $F\subseteq L$. It is readily seen that $F$ is a filter if  and only if  the characteristic map $L\lra\{0,1\}$ of $F$ preserves finite meets (including the empty one). 
This motivates the notion of d-filter of d-lattices.  We remind the reader that the Boolean algebra $\mathbb{B}=\{0,1,\dbt,\dbf\}$ is viewed as a d-lattice with ${\bf con}=\{0,\dbt,\dbf\}$ and ${\bf tot}=\{1,\dbt,\dbf\}$. 

\begin{defn} 
Let $\mathcal{L}=(L;\dbt,\dbf;{\bf con},{\bf tot})$ be a d-lattice. A d-filter of $\mathcal{L}$ is a map $f\colon L\lra \bbB$ subject to the following conditions: \begin{enumerate}[label=\rm(\alph*)] \setlength{\itemsep}{0pt}    
\item  $f(\dbt)\sqsupseteq  \dbt$, $f(\dbf)\sqsupseteq \dbf$.  \item $f$ preserves  {\bf tot}
and finite meets; that is, $f({\bf tot})\subseteq\{1,\dbt,\dbf\}$  and $f(x\sqcap y)=f(x)\sqcap f(y)$ for all $x,y\in L$.  
\end{enumerate}  
\end{defn}

A \emph{d-ideal} of a d-lattice  $\mathcal{L}=(L;\dbt,\dbf;{\bf con},{\bf tot})$ is defined to be a d-filter of the order-dual $\mathcal{L}^\partial$ (Example \ref{order dual}) of $\mathcal{L}$. Explicitly, a d-ideal of $\mathcal{L}$ is a map $h\colon L\lra \bbB$ such that \begin{enumerate}[label=\rm(\alph*)] \setlength{\itemsep}{0pt}  
\item $h(\dbf)\sqsupseteq \dbt$, $h(\dbt)\sqsupseteq  \dbf$; and 
\item $h$ maps  {\bf con} to {\bf tot} and transforms finite joins to meets; that means,  $h({\bf con})\subseteq\{1,\dbt,\dbf\}$ and $h(x\sqcup y)=h(x)\sqcap h(y)$ for all $x,y\in L$.  \end{enumerate}  

If $h\colon L\lra \bbB$ is a d-ideal of $\mathcal{L}$, the complement of $h$, which is defined to be the map $$\neg h\colon L\lra \bbB, \quad x\mapsto \neg (h(x)),$$ is called a \emph{co-d-ideal} of $\mathcal{L}$, where $\neg\colon\bbB\lra\bbB$ is the negation operator on the Boolean algebra $\bbB$ given by $\neg 0=1$, $\neg 1=0$, $\neg\dbt=\dbf$ and $\neg\dbf=\dbt$.
It is easily verified that a map $g\colon L\lra \bbB$ is a co-d-ideal of a d-lattice $(L;\dbt,\dbf;{\bf con},{\bf tot})$ if  and only if  it satisfies the following conditions: \begin{enumerate}[label=\rm(\alph*)] \setlength{\itemsep}{0pt}  
\item $g(\dbt)\sqsubseteq  \dbt$, $g(\dbf)\sqsubseteq \dbf$. 
\item $g$ preserves  {\bf con}  and finite joins; that is,  $g({\bf con})\subseteq\{0,\dbt,\dbf\}$  and $g(x\sqcup y)=g(x)\sqcup g(y)$ for all $x,y\in L$.  \end{enumerate}   
  
The d-ideals of a d-lattice correspond bijectively to its co-d-ideals by taking complement. So, in this paper we shall talk about structures of co-d-ideals instead of d-ideals.

\begin{prop} Let $f\colon L\lra\bbB$ be a d-filter of a d-lattice $\mathcal{L}=(L;\dbt,\dbf;{\bf con},{\bf tot})$. Then the following are equivalent: \begin{enumerate}[label=\rm(\arabic*)] \setlength{\itemsep}{0pt}  \item $f\colon \mathcal{L}\lra\bbB$ is a d-lattice homomorphism.
\item $f$ preserves  {\bf con}  and finite joins. 
\item $f$ is a co-d-ideal.  \end{enumerate}  \end{prop}

\begin{defn} 
Let $\mathcal{L}=(L;\dbt,\dbf;{\bf con},{\bf tot})$ be a d-lattice. A map $f\colon L\lra \bbB$ is said to be a prime d-filter  of $\mathcal{L}$ if it is at the same time a d-filter and a co-d-ideal of $\mathcal{L}$. 
In other words, a prime d-filter of  $\mathcal{L}$ is  a d-lattice homomorphism $f\colon\mathcal{L}\lra \bbB$. \end{defn}

A d-filter $f$ of a d-lattice is said to be \emph{proper} if $f(0)=0$. It is easily verified that a d-filter $f$ is proper if and only if it preserves $\dbt$ and $\dbf$; that is,   $f(\dbt)=\dbt$ and $f(\dbf)=\dbf$. Dually, a d-ideal $h$  is  \emph{proper} if $h(1)=0$. 
It is well-known that each maximal proper filter of a distributive lattice is a prime filter, but, even maximal proper d-filters of a d-Boolean algebra  need not be prime.

\begin{exmp}   Consider the d-Boolean algebra $([0,1]\times[0,1]; (1,0),(0,1);{\bf con},{\bf tot})$, where ${\bf con}=\{(a,b):a\leq 1-b\}$ and ${\bf tot}=\{(a,b):1-a\leq b\}$.   \begin{center}
\begin{tikzpicture}[scale=0.9] \fill[lightgray] (0,3) -- (3,3) -- (3,0) -- cycle;
\draw[thick] (0,0) -- (3,0) -- (0,3) -- cycle;
\draw[thick] (0,3) -- (3,3) -- (3,0) -- cycle;
\node at (2,2) {{\bf tot}}; \node at (1,1) {{\bf con}}; 
\node at (3.3,0) {$\dbt$}; \node at (-0.3,3) {$\dbf$};
\end{tikzpicture} \end{center} 
The  map $$ f\colon [0,1]\times [0,1]\lra \bbB, \quad f(a,b)=\begin{cases}0,& a=0,~b=0 \\ \dbt &a>0,~b=0 \\ \dbf &a=0,~b>0 \\ 1 & a>0,~b>0 \end{cases}$$  is a maximal proper d-filter. But, $f$ is not a prime d-filter since it does not preserve {\bf con}.  
\end{exmp} 

For each d-filter $f$ of a d-lattice $(L;\dbt,\dbf;{\bf con},{\bf tot})$, let   $$ 
F_+=\{a\in L_+:f(a\sqcup\dbf)=1\},\quad F_-=\{b\in L_-:f(\dbt\sqcup b)=1\}.$$ Then $F_+$ is a filter of $L_+$  and $F_-$ is a filter of $L_-$. Since $a\sqcup b= (a\sqcup\dbf)\sqcap(\dbt\sqcup b)$ for all $a\in L_+$ and $b\in L_-$,  it is readily seen that the pair $(F_+,F_-)$ of filters satisfies that $$ a\sqcup b\in{\bf tot}\implies \text{either $a\in F_+$ or $b\in F_-$} $$  for all $a\in L_+$ and $b\in L_-$. Conversely, we have the following 

\begin{prop} \label{structure of d-filter} Suppose $\mathcal{L}=(L;\dbt,\dbf;{\bf con},{\bf tot})$ is a d-lattice;   $J_+$ is a filter of the lattice $L_+$; and $J_-$ is a filter of the lattice $L_-$.  If for all $a\in L_+$ and $b\in L_-$ it holds that  $$ a\sqcup b\in{\bf tot}\implies \text{either $a\in J_+$ or $b\in J_-$},$$  then the map $f\colon L\lra\mathbb{B}$, given by  $$f(a\sqcup b)= \begin{cases}1 & a\in J_+,~b\in J_-,\\
\dbt & a\in J_+,~b\notin J_-,\\ 
\dbf & a\notin J_+,~b\in J_-,\\
0 &a\notin J_+,~b\notin J_- \end{cases}$$ for all $a\in L_+$ and $b\in L_-$, is the unique d-filter of $\mathcal{L}$ such that  $F_+=J_+$ and $F_-=J_-$. \end{prop}

\begin{proof} Straightforward verification. \end{proof}

\begin{cor}\label{proper d-filter} If $f\colon L\lra \bbB$ is a d-filter of a d-lattice $(L;\dbt,\dbf;{\bf con},{\bf tot})$,   then  $$f(a\sqcup b)=f(a)\sqcup f(b)$$ for all $a\in L_{+}$ and $b\in L_{-}$.    \end{cor}

\begin{proof} This follows from the fact that the d-filter determined by $F_+=\{a\in L_+:f(a\sqcup\dbf)=1\}$ and $F_-=\{b\in L_-:f(\dbt\sqcup b)=1\}$ as in Proposition \ref{structure of d-filter}  possesses that property.
 \end{proof}
 
For each co-d-ideal $g$ of a d-lattice $(L;\dbt,\dbf;{\bf con},{\bf tot})$, let $$ G_+=\{a\in L_+:g(a)=0\},\quad G_-=\{b\in L_-:g(b)=0\}.$$  Then  $G_+$ is an ideal of the lattice $L_+$ and $G_-$ is an ideal of the lattice $L_-$ such that  for all $a\in L_+$ and $b\in L_-$,  $$ a\sqcup b\in{\bf con}\implies \text{either $a\in G_+$ or $b\in G_-$}. $$ Parallel to Proposition \ref{structure of d-filter} we have: 

\begin{prop} \label{structure of d-ideal} Suppose $\mathcal{L}=(L;\dbt,\dbf;{\bf con},{\bf tot})$ is a d-lattice;  $I_+$ is an ideal of the lattice $L_+$; and $I_-$ is an ideal of the lattice $L_-$. If for all $a\in L_+$ and $b\in L_-$ it holds that  $$ a\sqcup b\in{\bf con}\implies \text{either $a\in I_+$ or $b\in I_-$}, $$ then the map $g\colon L\lra\mathbb{B}$, given by  $$g(a\sqcup b)  =\begin{cases} 1 & a\notin I_+,~b\notin I_-,\\
\dbt & a\notin I_+,~b\in I_-,\\ 
\dbf & a\in I_+,~b\notin I_-,\\
0 & a\in I_+,~b\in I_- \end{cases}$$  for all $a\in L_+$ and $b\in L_-$, is the unique co-d-ideal of $\mathcal{L}$ such that $G_+=I_+$ and $G_-=I_-$.  \end{prop}
  
\begin{prop} Let $\mathcal{L}=(L;\dbt,\dbf;{\bf con},{\bf tot})$ be a d-lattice. If $f$ is a d-filter  and $g$ is a co-d-ideal of $\mathcal{L}$ such that  $f\sqsubseteq g$, then there is a prime d-filter $h$ of $\mathcal{L}$ for which $f\sqsubseteq h\sqsubseteq g$. \end{prop}

\begin{proof} First of all, notice that $f\sqsubseteq g$ implies $f(\dbt)=\dbt=g(\dbt)$ and $f(\dbf)=\dbf=g(\dbf)$. Then, the ideal  $G_+=\{a\in L_+:g(a)=0\}$ is disjoint with the filter $ 
F_+=\{a\in L_+:f(a\sqcup\dbf)=1\}$; the ideal  $G_-=\{b\in L_-:g(b)=0\}$ is disjoint with the filter $F_-=\{b\in L_-:f(\dbt\sqcup b)=1\}.$ 

Pick an ideal $I_+$ of $L_+$ that contains $G_+$ and is maximal with respect to being disjoint with $F_+$; pick an ideal $I_-$ of  $L_-$ that contains $G_-$ and is maximal with respect to being disjoint with $F_-$. Maximality of $I_+$ implies that $I_+$ is a prime ideal of $L_+$, hence $K_+\coloneqq L_+\setminus I_+$ is a filter of $L_+$. Likewise, $K_-\coloneqq L_-\setminus I_-$ is a filter of $L_-$. 

Since $G_+\subseteq I_+$ and $G_-\subseteq I_-$,   for all $a\in L_+$ and $b\in L_-$ it holds that $$ a\sqcup b\in{\bf con}\implies \text{either $a\in I_+$ or $b\in I_-$}.$$ Since $F_+\subseteq K_+$ and $F_-\subseteq K_-$,   for all $a\in L_+$ and $b\in L_-$ it holds that  $$ a\sqcup b\in{\bf tot}\implies \text{either $a\in K_+$ or $b\in K_-$}.$$ Therefore, the pair $(I_+,I_-)$ of ideals determines a co-d-ideal, say $h$; the pair $(K_+,K_-)$ of filters determines a d-filter, say $k$. It is readily seen that $h=k$, so $h$ is a prime d-filter. That $f\sqsubseteq h\sqsubseteq g$ is clear by the construction of $h$.
\end{proof}
  
For d-Boolean algebras we can say more. Let $\mathcal{L}=(L;\dbt,\dbf;{\bf con},{\bf tot})$ be a d-Boolean algebra. Since $x\sqcup x^\dag\in{\bf con}\cap{\bf tot}$ for all $x\in L_+\cup L_-$, then for each filter $F_+$ of $L_+$ and each filter $F_-$ of $L_-$, the following are equivalent: 
\begin{itemize} \setlength{\itemsep}{0pt} 
\item the pair $(F_+,F_-)$ determines a  d-filter (as in Proposition \ref{structure of d-filter}); 
\item  $L_+=F_+\cup F_-^\dag$, where $F_-^\dag= \{x^\dag: x\in F_-\}$; 
\item $L_-=F_+^\dag\cup F_-$, where  $F_+^\dag= \{x^\dag: x\in F_+\}$. \end{itemize} 
Likewise, for each ideal $I_+$ of $L_+$ and each ideal $I_-$ of $L_-$, the following are equivalent: 
\begin{itemize} \setlength{\itemsep}{0pt} 
\item the pair $(I_+,I_-)$ determines a  co-d-ideal of $\mathcal{L}$  (as in Proposition \ref{structure of d-ideal}); 
\item    $L_+=I_+\cup I_-^\dag$, where $I_-^\dag= \{x^\dag: x\in I_-\}$; 
\item   $L_-=I_+^\dag\cup I_-$, where $I_+^\dag= \{x^\dag: x\in I_+\}$. \end{itemize} 

The following proposition implies, in particular, that prime d-filters of a d-Boolean algebra are pairwise incomparable.
\begin{prop}\label{dBA-IF} Let $\mathcal{L}=(L;\dbt,\dbf;{\bf con},{\bf tot})$ be a d-Boolean algebra.  If $f$ is a  d-filter and $g$ is a  co-d-ideal  of $\mathcal{L}$ such that $f\sqsubseteq g$, then $f=g$. \end{prop}

\begin{proof} Let $$G_+=\{a\in L_+:g(a)=0\}, \quad G_-=\{b\in L_-:g(b)=0\};$$     $$F_+=\{a\in L_+:f(a\sqcup\dbf)=1\},\quad F_-=\{b\in L_-:f(\dbt\sqcup b)=1\}.$$ The inequality $f\sqsubseteq g$ ensures that both $G_+$ and $G_-$ are proper ideals, both $F_+$ and $F_-$ are proper filters, and that $F_+\cap G_+=\emptyset$, $F_-\cap G_-=\emptyset$. 
Since $\mathcal{L}$ is a d-Boolean algebra, then $L_+= F_+\cup F_-^\dag$ and $L_- = G_-\cup G_+^\dag$. From $F_-\cap G_-=\emptyset$ and $L_- = G_-\cup G_+^\dag$ it follows that $F_-\subseteq L_-\setminus G_-\subseteq G_+^\dag$. From $F_+\cap G_+=\emptyset$ and $L_+= F_+\cup F_-^\dag$ it follows that $G_+\subseteq L_+\setminus F_+\subseteq F_-^\dag$, hence  $G_+^\dag\subseteq   F_-$ because $^\dag$ is an order-reversing isomorphism. Therefore, $F_-= G_+^\dag$ and consequently, $F_-= L_-\setminus G_-$ and $F_+= L_+\setminus G_+$.  Then,  by Proposition \ref{structure of d-ideal} and Proposition \ref{structure of d-filter} we obtain that $g=f$.
\end{proof}

The following lemma echoes the fact that if $F$ is a prime filter of a Boolean algebra $A$, then for each $a$ of $A$,   either $a$ or its  complement belongs to $F$, but not both. The proof will use the fact that in a d-Boolean algebra, $$a\sqcup b \in {\bf con}\iff b\sqsubseteq a^\dag \iff b^\dag\sqcup a^\dag\in  {\bf tot} $$ for all $a\in L_{+}$ and $b\in L_{-}$,  
which follows from that the consistency predicate and the totality predicate of a d-Boolean algebra are given by ${\bf con}=\{a\sqcup b: a\in L_+, b\in  L_-, a^\dag\sqsupseteq b \}$ and ${\bf tot}=\{a\sqcup b: a\in L_+, b\in  L_-,a^\dag\sqsubseteq b\}$.   
   
\begin{lem}\label{prime d-filter} 
A proper d-filter  $f\colon L\lra \bbB$ of a d-Boolean algebra $(L;\dbt,\dbf;{\bf con},{\bf tot})$  is prime if and only if it  satisfies:  \begin{enumerate}[label=\rm(\alph*)] \setlength{\itemsep}{0pt}  \item For all $a\in L_+$, $f(a)=\dbt \iff f(a^\dag)=0$; \item For all $b\in L_-$, $f(b)=\dbf \iff f(b^\dag)=0$.  \end{enumerate}   \end{lem}

\begin{proof} We prove the necessity first.   To see (a), let $a\in L_+$. Since $f$  preserves both {\bf con} and {\bf tot}, it follows that $f({\bf con}\cap{\bf tot})\subseteq\{\dbt,\dbf\}$, particularly $f(a\sqcup a^\dag)\in \{\dbt,\dbf\}$. Since $f(a)\in\{0,\dbt\}$, $f(a^\dag)\in\{0,\dbf\}$ and $f(a\sqcup a^\dag)= f(a)\sqcup f(a^\dag)$, it follows that $f(a)=\dbt\iff f(a^\dag)=0$. That $f$ satisfies (b) can be verified in a similar way.

For the  sufficiency, we only need to show that $f$ is a co-d-ideal.  Let $$F_+=\{a\in L_+:f(a\sqcup\dbf)=1\},\quad F_-=\{b\in L_-:f(\dbt\sqcup b)=1\}. $$ Then $F_+$ is a filter  of $L_+$,  $F_-$ is a filter of $L_-$, and $f$ is determined by the pair $(F_+,F_-)$ as in Proposition \ref{structure of d-filter}.

Since $f$ is a proper d-filter, it preserves $\dbt$ and $\dbf$. Then for each $a\in L_+$, it holds that $f(a\sqcup\dbf)=f(a)\sqcup f(\dbf)=f(a)\sqcup\dbf$ by Corollary \ref{proper d-filter},  which entails that $ f(a\sqcup\dbf)=1$  if and only if $f(a)=\dbt$. Then by the equivalence in (a) we obtain that $a\in F_+\iff f(a^\dag)=0 $ for all $a\in L_+$.
Since $(-)^\dag\colon L_+\lra L_-$ is an order-reversing isomorphism, it follows that $$G_-\coloneqq \{a^\dag:a\in F_+\}= \{b\in L_-: f(b)=0\}$$  is an ideal of $L_-$.  Likewise,  
$$G_+\coloneqq \{b^\dag:b\in F_-\}=\{a\in L_+:f(a)=0\} $$ is an ideal of $L_+$.   

We claim that the pair $(G_+,G_-)$ of ideals satisfies that  for all $a\in L_+$ and $b\in L_-$, $$ a\sqcup b\in{\bf con}\implies a\in G_+~\text{or}~b\in G_-.$$   
Suppose on the contrary that neither $a\in G_+$ nor $b\in G_-$. Then $f(a)=\dbt$ and $f(b)=\dbf$, hence $f(a^\dag)=0$ and $f(b^\dag)=0$ by assumption, and consequently $f(b^\dag\sqcup a^\dag)=0$ (Corollary \ref{proper d-filter}). This contradicts that $b^\dag\sqcup a^\dag\in{\bf tot}$ and that $f$ preserves {\bf tot}.   

Proposition \ref{structure of d-ideal} guarantees that  the pair $(G_+,G_-)$ of ideals determines a co-d-ideal $g$. It is readily verified that $g=f$, so $f$ is a co-d-ideal. \end{proof}
 
\begin{prop}Let $\mathcal{L}=(L;\dbt,\dbf;{\bf con},{\bf tot})$ be a d-Boolean algebra. \begin{enumerate}[label=\rm(\roman*)] \setlength{\itemsep}{0pt}
\item If $g$ is a prime d-filter of $\mathcal{L}$, then   $G_+=\{a\in L_+:g(a)=0\}$ is a prime ideal of $L_+$,   $G_-=\{b\in L_-:g(b)=0\}$ is a prime ideal of $L_-$,  and $G_+, G_-$ determine each other via $G_+^\dag=L_-\setminus G_-$. 
\item  If $I_+$ is a prime ideal of the lattice $L_+$, then there is a unique prime d-filter $g$ of  $\mathcal{L}$ such that $I_+=  \{a\in L_+: g(a)=0\}$. \end{enumerate} Therefore, the prime d-filters of $\mathcal{L}$ correspond bijectively to the prime ideals of the lattice $L_+$. \end{prop}

\begin{proof}(i)  
We only need to check that $G_+^\dag=L_-\setminus G_-$. For each $b\in L_-$, by Lemma \ref{prime d-filter}  either $b$ belongs to $G_-$ or $b^\dag$ belongs to $G_+$, but not both. It follows that $G_+^\dag=L_-\setminus G_-$.

(ii) Since $I_+$ is a prime ideal of $L_+$, then $L_+\setminus I_+$ is a prime filter of $L_+$, hence  $I_-\coloneqq \{a^\dag: a\in L_+\setminus I_+\}$ is a prime ideal of $L_-$. The co-d-ideal determined by the ideals $I_+$ and $I_-$ (as in Proposition \ref{structure of d-ideal}) is also a d-filter,  hence it is the unique prime d-filter  satisfying the requirement. \end{proof} 

Now we introduce the spectra of d-lattices.
For each d-lattice $\mathcal{L}=(L;\dbt,\dbf;{\bf con},{\bf tot})$, let $$\dSpec\mathcal{L} $$ be the set of all prime d-filters of $\mathcal{L}$. Make $\dSpec\mathcal{L} $ into a bitopological space by considering as  $\tau_+$ the topology generated by the collection of $$\phi_+(a)=\{g\in\dSpec\mathcal{L}:g(a)=\dbt\}, \quad a\in L_+; $$ and as  $\tau_-$ the topology generated by the collection of $$\phi_-(b)=\{g\in\dSpec\mathcal{L}:g(b)=\dbf\}, \quad b\in L_-. $$
The bitopological space $$(\dSpec\mathcal{L},\tau_+,\tau_-)$$ is called the \emph{spectrum} of the d-lattice $\mathcal{L}$. It is not hard to check that  each open set  of $\tau_+$ is of the form $$\phi_+(I_+)\coloneqq  
\{g\in\dSpec\mathcal{L}: \exists a\in I_+,\thinspace g(a)=\dbt\}$$ for some ideal $I_+$ of the lattice $L_+$;  likewise for $\tau_-$. In particular,  for all $a\in L_+$ and $b\in L_-$, $\phi_+(a)$ is a compact open set of $(\dSpec\mathcal{L},\tau_+)$, $\phi_-(b)$ is a compact open set of $(\dSpec\mathcal{L},\tau_-)$. 

Assigning to  each d-lattice its spectrum defines a contravariant functor $$\dSpec\colon{\bf dLat}^{\rm op}\lra{\bf BiTop}.$$

\begin{exmp}  The spectrum $\dSpec\bbB$ of the d-lattice $\bbB$   is a singleton bitopological space. \end{exmp}

\begin{prop}\label{spectrum as composite} The spectrum of a d-lattice $\mathcal{L}=(L;\dbt,\dbf;{\bf con},{\bf tot})$ is precisely the bitopological space of d-points of the d-frame of ideals of  $\mathcal{L}$; that is,  $\dSpec\mathcal{L} = \dpt\circ\idl\mathcal{L}$. \end{prop}

\begin{proof} By Proposition \ref{idl left adjoint to forgetful}, the functor $\idl\colon\dLat \longrightarrow \dFrm$ is left adjoint to the forgetful functor $U\colon \dFrm\longrightarrow\dLat$, it follows that a d-point of $\idl\mathcal{L}$ is precisely a d-lattice homomorphism $\mathcal{L}\lra\bbB$, hence a prime d-filter of $\mathcal{L}$.  It is routine to check that the bitopological structure of $\dSpec\mathcal{L}$ coincides with that of $\dpt\circ\idl\mathcal{L}$,   therefore $\dSpec\mathcal{L} =\dpt\circ\idl\mathcal{L}$. \end{proof}

For each d-lattice $\mathcal{L}=(L;\dbt,\dbf;{\bf con},{\bf tot})$, since the spectrum $(\dSpec\mathcal{L},\tau_+,\tau_-)$ coincides with the bitopological space of d-points of  $\idl\mathcal{L}$,  the map  $$\varphi\colon \idl L\lra \tau_+\times \tau_-, \quad (I_+,I_-) \mapsto (\phi_+(I_+), \phi_-(I_-)) $$  is a surjective d-frame homomorphism   $\idl\mathcal{L}\lra \dO\circ\dSpec\mathcal{L}$, it is indeed the component at $\idl\mathcal{L}$ of the counit of the adjunction  $\dO\dashv\dpt$. The d-frame $\idl\mathcal{L}$ is spatial if and only if $\varphi$ is an isomorphism of d-frames. But, $\varphi$ is not always an isomorphism of d-frames; that means, the d-frame  of ideals of a d-lattice need not be spatial. 

\begin{exmp}\label{idl is not always d-spatial} There is a d-lattice $\mathcal{L}$ for which the d-frame $\idl\mathcal{L}$ is not spatial. The example is related to the d-lattice $\omega(\bbB)$, so we temporarily write  $\top$ and $\bot$ for the complementary pair of the Boolean algebra $\bbB$, and reserve the symbols $\dbt$ and $\dbf$ for the d-lattice to be constructed.
Consider the structure $\mathcal{L}=(L;\dbt,\dbf;{\bf con},{\bf tot})$, where \begin{itemize} \setlength{\itemsep}{0pt}
\item   $L=\mathbb{B}\times\mathbb{B}$;   \item $\dbt=(1,0)$, $\dbf=(0,1)$; \item ${\bf con}=\{(a,b)\in \mathbb{B}\times \mathbb{B}: a\sqcap b=0\}$;  \item ${\bf tot}=\{(a,b)\in\mathbb{B}\times\mathbb{B}:a=1 ~ \text{or} ~ b=1\}\cup\{(\top,\bot) \}$. \end{itemize} It is readily verified that $\mathcal{L}$ is a d-lattice. The difference between the d-lattice $\mathcal{L}$ and the d-Boolean algebra $\omega(\bbB)$ is that the element $(\bot,\top)$ of $\bbB\times\bbB$ belongs to the totality predicate of  $\omega(\bbB)$, but not to that of $\mathcal{L}$. We show in three steps that the d-frame $\idl\mathcal{L}$ is not spatial. 

{\bf Step 1}.   $\dSpec\mathcal{L}=\dSpec \omega(\mathbb{B})$.  Both $\mathcal{L}$ and $\omega(\mathbb{B})$ have 2 prime d-filters, which are determined as in Proposition \ref{structure of d-filter} by the pair $(\uparrow\!\bot, \uparrow\!\bot)$ and  the pair $(\uparrow\!\top, \uparrow\!\top)$ of filters of $\bbB$.  Then one readily verifies that the bitopological spaces $\dSpec\mathcal{L}$ and  $\dSpec\omega(\mathbb{B})$ are equal to each other. 
 
{\bf Step 2}. Since the pair $(\downarrow\!\bot, \downarrow\!\top)$ belongs to the totality predicate of $\idl\omega(\mathbb{B})$, but not to that of $\idl\mathcal{L}$, the totality predicate of $\idl\mathcal{L}$ is a proper subset of the totality predicate of  $\idl\omega(\mathbb{B})$  which is  a finite set, so the d-frames $\idl\omega(\mathbb{B})$ and $\idl\mathcal{L}$ are not isomorphic.   

{\bf Step 3}. 
The d-frame $\idl\omega(\mathbb{B})$ is spatial by Example \ref{omega(L) is spatial} , so it is the d-frame of open sets of the bitopological space $\dSpec\omega(\mathbb{B})$. Since $\dSpec\mathcal{L} = \dpt\circ \idl\mathcal{L}$ by Proposition \ref{spectrum as composite}, it follows that  $\dO\circ\dpt(\idl\mathcal{L})  = \dO \dSpec\mathcal{L} = \dO \dSpec\omega(\bbB) =  \idl\omega(\mathbb{B}),$    therefore the d-frame $\idl\mathcal{L}$  is not spatial, otherwise it would be isomorphic to $\idl\omega(\mathbb{B})$. \end{exmp}   
   
However, the d-frame  of ideals of each d-Boolean algebra is always spatial.
\begin{prop}\label{compact and zero-dimensional d-frame is spatial} 
If $\mathcal{L}$ is a d-Boolean algebra, then  $\varphi\colon\idl\mathcal{L}\lra \dO\circ\dSpec\mathcal{L}$ is an isomorphism of d-frames. In particular, the d-frame $\idl\mathcal{L}$ is spatial. 
\end{prop}

\begin{proof} If  $\mathcal{L}$ is a d-Boolean algebra, then the d-frame $\idl\mathcal{L}$ is compact and d-zero-dimensional by Theorem \ref{dBA vs dfrm},
hence compact and regular in the sense of  \cite[Section 6]{JM2006} and consequently,   spatial by  \cite[Theorem 6.11]{JM2006}. By Proposition \ref{spectrum as composite} we have $\dSpec\mathcal{L} = \dpt\circ\idl\mathcal{L}$, it follows that $\idl\mathcal{L}$ is the d-frame of open sets of the bitopological space $\dSpec\mathcal{L}$,  so the map $\varphi$ is an isomorphism of d-frames.  \end{proof}  
 
\begin{cor} \label{duality between zk dframes and zk spaces}
The duality between spatial d-frames and d-sober bitopological spaces cuts down to a duality between compact and d-zero-dimensional d-frames and d-Boolean spaces.
\end{cor}

\begin{proof} On the one hand, every d-Boolean space is order-separated, hence d-sober by \cite[Theorem 3.9]{JM2008}. On the other hand, Proposition \ref{other side of the duality} shows that every compact and d-zero-dimensional d-frame is the d-frame $\idl\mathcal{L}$ of ideals of some d-Boolean algebra $\mathcal{L}$, hence  spatial by  Proposition \ref{compact and zero-dimensional d-frame is spatial}.  Thus, the conclusion follows from the fact that a bitopological space $(X,\tau_+,\tau_-)$ is compact and zero-dimensional if and only if   the d-frame  $\dO(X, \tau_+, \tau_-)$ is compact and d-zero-dimensional, as observed in \cite[Section 2.3]{Jakl}. \end{proof}
 
\begin{thm} \label{duality of Stone bitop}
The category of d-Boolean algebras is dually equivalent to the category of d-Boolean spaces.
\end{thm}

\begin{proof} Theorem \ref{dBA vs dfrm} shows that there is an equivalence  between the category of d-Boolean algebras and the category of compact and d-zero-dimensional d-frames. Corollary  \ref{duality between zk dframes and zk spaces} says that there is a dual equivalence  between the category of compact and d-zero-dimensional d-frames and the category of d-Boolean spaces. The composite of these equivalences gives the desired duality  $$\bfig
\morphism(0,0)|a|/@{->}@<3pt>/<750,0>[{\bf dBA}^{\rm op}`{\bf dBoolSp};\dSpec ]
\morphism(0,0)|b|/@{<-}@<-3pt>/<750,0>[{\bf dBA}^{\rm op}`{\bf dBoolSp}; \dClop] \efig$$ between d-Boolean algebras and d-Boolean spaces. \end{proof}
 
\begin{rem}\label{Stone-type} The Boolean algebra $\bbB$ is a \emph{schizophrenic object} \cite[Chapter VI]{Johnstone} (also called a \emph{dualizing object} \cite[page 14]{JM2006})  for the duality in Theorem \ref{duality of Stone bitop}. So, the duality is an example of \emph{natural dualities}, a topic pioneered by Priestley, see e.g. \cite{CD1998}.

On the one hand, since a prime d-filter of a d-Boolean algebra is a homomorphism to $\bbB$, the functor $\dSpec$ sends each d-Boolean algebra to the space of homomorphisms  to $\bbB$. 

On the other hand, make the Boolean algebra $\bbB$ into a bitopological space by considering as $\tau_+$ the topology generated by $\{\{\dbt,1\}, \{0,\dbt\}\}$, and as $\tau_-$ the topology generated by $\{\{\dbf,1\},\{0,\dbf\}\}$. Then for each bitopological space $(X,\tau_+,\tau_-)$ and each continuous map $f\colon X\lra \bbB$, the set $$U\coloneqq\{x\in X:f(x)\sqsupseteq\dbt\}$$ is $\tau_+$-open and $\tau_-$-closed; the set  $$V\coloneqq\{x\in X:f(x) \sqsupseteq\dbf\}$$ is $\tau_-$-open and $\tau_+$-closed; and the pair $(U,V)$ is an element of the d-Boolean algebra of d-clopen sets of $(X,\tau_+,\tau_-)$.   Furthermore, if $U$ is $\tau_+$-open and $\tau_-$-closed, and if $V$ is  $\tau_-$-open and $\tau_+$-closed, then  $$f_{U,V}\colon X\lra\bbB,\quad f_{U,V}(x)=\begin{cases}1 &x\in U\cap V, \\ \dbt &x\in U\setminus V,\\ \dbf & x\in V\setminus U,\\ 0 &x\notin U\cup V\end{cases}$$ is a continuous map. The correspondence $(U,V)\mapsto f_{U,V}$ is a bijection. So, the functor $\dClop$ sends each bitopological space to the algebra of continuous maps to $\bbB$. \end{rem} 
 
It is known that the spectrum of  a distributive lattice $L$ is Hausdorff if and only if $L$ is a Boolean algebra, see e.g. \cite[page 71]{Johnstone}. By Theorem \ref{duality of Stone bitop}, the spectrum $\dSpec\mathcal{L}$ of each d-Boolean algebra is a d-Boolean space,   hence order-separated. But the converse is not true. The d-lattice $\mathcal{L}$ in Example \ref{idl is not always d-spatial} is not a d-Boolean algebra, but its spectrum, which coincides with that of the d-Boolean algebra $\omega(\bbB)$, is   order-separated.  However, there is a partial converse.

\begin{prop}\label{d-Boolean and order-separated}
Let $\mathcal{L}=(L;\dbt,\dbf;{\bf con},{\bf tot})$ be a d-lattice. If the d-frame $\idl \mathcal{L}$ is  spatial and the bitopological space $\dSpec\mathcal{L}$ is order-separated, then $\mathcal{L}$ is a d-Boolean algebra.
\end{prop}

We need a lemma,  contained in Jakl \cite[Lemma 4.1.29]{Jakl},  which echoes the fact that every compact subset of a Hausdorff space is closed. A subset of a topological space $(X,\tau)$ is \emph{saturated} \cite{Gierz2003} if it is an upper set with respect to the specialization order. 

\begin{lem}\label{compact opens and dclopens} {\rm (\cite{Jakl})} Let $(X,\tau_+,\tau_-)$ be an order-separated bitopological space. \begin{enumerate}[label=\rm(\roman*)] \setlength{\itemsep}{0pt} \item Every compact saturated set of $(X,\tau_+)$ is a closed set of $(X,\tau_-)$. In particular, every compact open set of $(X,\tau_+)$ is a closed set of $(X,\tau_-)$. 
\item  Every compact saturated set of $(X,\tau_-)$  $H$ is a closed set of $(X,\tau_+)$. In particular, every compact open set of $(X,\tau_-)$ is a closed set of $(X,\tau_+)$. \end{enumerate} \end{lem}  
 
\begin{proof}[Proof of Proposition \ref{d-Boolean and order-separated}] Let $a\in L_+$. Since $(\dSpec\mathcal{L},\tau_+,\tau_-)$ is order-separated, $\phi_+(a)$ is a compact open set of $(\dSpec\mathcal{L},\tau_+)$, it follows from Lemma \ref{compact opens and dclopens} that $\phi_+(a)$ is a closed set of $(\dSpec\mathcal{L},\tau_-)$, so $\phi_+(a)$ is d-complemented in the d-frame  of open sets of $(\dSpec\mathcal{L},\tau_+,\tau_-)$. Since the d-frame $\idl\mathcal{L}$  is spatial by assumption, then   $$\varphi\colon \idl L\lra \tau_+\times \tau_-, \quad (I_+,I_-)\mapsto (\phi_+(I_+), \phi_-(I_-)) $$  is   an isomorphism of d-frames $\idl\mathcal{L}\lra \dO\circ\dSpec\mathcal{L}$. It follows that the  ideal $\downarrow\!a$ of $L_+$, which corresponds to the $\tau_+$-open and $\tau_-$-closed set $\phi_+(a)$, is d-complemented in the d-frame $\idl\mathcal{L}$, then $a$ is d-complemented by Lemma \ref{d-complemented ideal}. Likewise, each $b$ of $L_-$ is d-complemented. Therefore, $\mathcal{L}$ is a d-Boolean algebra.
\end{proof}
 
We end this section with a duality between complete d-Boolean algebras and extremally disconnected d-Boolean spaces.

\begin{defn} (\cite[Definition 2.2]{Datta}) A bitopological space $(X,\tau_+,\tau_-)$ is  extremally disconnected provided that the $\tau_-$-closure of each $\tau_+$-open set is   $\tau_+$-open, and that the $\tau_+$-closure of each $\tau_-$-open set is $\tau_-$-open. \end{defn}

Let $(X,\tau_+,\tau_-)$ be a bitopological space and $U\in\tau_+$. It is not hard to see that, in the d-frame $\dO(X,\tau_+,\tau_-)$,  the d-pseudo-complement of $U$ is given by the complement of the closure of $U$ in $(X,\tau_-)$; that is, $U^*=X\setminus \cl_-U$. By help of this fact one sees that $(X,\tau_+,\tau_-)$ is extremally disconnected if and only if the d-frame $\dO(X,\tau_+,\tau_-)$ is extremally disconnected. Then, it follows from Lemma \ref{complete dBA vs extremal discon} that a zero-dimensional bitopological space is extremally disconnected if and only if its d-Boolean algebra of  d-clopen sets is complete.   Combining this  with Theorem \ref{duality of Stone bitop} yields: 
\begin{prop}
The category of complete d-Boolean algebras and d-lattice homomorphisms is dually equivalent to the category of extremally disconnected d-Boolean spaces and continuous maps.
\end{prop}

 \section*{Acknowledgement}

The authors gratefully thank the referee for her/his  thorough analysis of the paper and insightful comments and suggestions which lead to a significant improvement of the paper.


\begin{thebibliography}{[99]} \setlength{\itemsep}{0pt}

 \bibitem{AHS}J. Ad\'{a}mek, H. Herrlich, G.E. Strecker, \emph{Abstract and Concrete Categories}, John Wiley and Sons,  1990.

\bibitem{BBGK} G. Bezhanishvili, N. Bezhanishvili, D. Gabelaia,   A. Kurz,  Bitopological duality for distributive lattices and Heyting algebras,  Mathematical Structures in Computer Science   20 (2010) 359-393. 

\bibitem{BK47} G. Birkhoff, S.A. Kiss, A ternary operation in distributive lattices, Bulletin of the American Mathematical Society  53 (1947) 749-752.

\bibitem{CD1998} D.M. Clark,   B.A. Davey, \emph{Natural Dualities for the Working Algebraist}, Cambridge University Press, Cambridge, 1998.
\bibitem{Datta}M.C. Datta, Projective bitopological spaces II, Journal of the Australian Mathematical Society  14 (1972)  119-128.

\bibitem{Engelking} R. Engelking, \emph{General Topology, Revised and Completed Edition}, Heldermann, Berlin, 1989. 

\bibitem{GG2024} M. Gehrke, S. van Gool, \emph{Topological Duality for Distributive Lattices, Theory and Applications}, Cambridge University Press, Cambridge, 2024.

\bibitem{Gierz2003} G. Gierz, K.H. Hofmann, K. Keimel, J.D. Lawson, M. Mislove,   D.S. Scott, \emph{Continuous Lattices and Domains}, Cambridge University Press, Cambridge, 2003.
 
\bibitem{HZ2025}J. He, D. Zhang, A Boolean-valued space approach to separation axioms and sobriety of bitopological spaces, Fuzzy Sets and Systems 503 (2025)  109241.
 
\bibitem{Jakl} T. Jakl,   \emph{d-Frames as Algebraic Duals of Bitopological Spaces}, Thesis, Charles University \& University of Birmingham, 2017. 

\bibitem{JJP} T. Jakl, A. Jung, A. Pultr,  Bitopology and four-valued logic, Electronic Notes in Theoretical Computer Science 325 (2016) 201-219.

\bibitem{Johnstone} P.T. Johnstone, \emph{Stone Spaces}, Cambridge University Press, Cambridge, 1982.

\bibitem{JM2006} A. Jung, M.A. Moshier, \emph{On the Bitopological Nature of Stone Duality}, Technical Report CSR-06-13, School of Computer Science, The University of Birmingham, 2006, 110 pages.  
	
\bibitem{JM2008} A. Jung, M.A.  Moshier,  A Hofmann-Mislove theorem for bitopological spaces, The Journal of Logic and Algebraic Programming 76 (2008) 161-174. 
 
\bibitem{Kelly} J.C. Kelly,   Bitopological spaces, Proceedings of the London Mathematical Society 13 (1963) 71-89.

\bibitem{Klinke} O.K. Klinke, \emph{A Bitopological Point-free Approach to Compactifications}, Thesis,   University of Birmingham, 2011.

\bibitem{KJM} O.K. Klinke, A. Jung, M.A.  Moshier, A bitopological point-free approach to compactifications, Topology and its Applications 158 (2011) 1551-1566.

\bibitem{Lal} S. Lal,  Pairwise concepts in bitopological spaces, Journal of the Australian Mathematical Society   26 (1978)  241-250.
 
\bibitem{MMW} M.A. Moshier, I. Mozo Carollo, J. Walters-Wayland,   On Isbell's density theorem for bitopological
pointfree spaces I, Topology and its Applications 273 (2020) 106962.

\bibitem{Mu}M.G. Murdeshwar, S.A. Naimpally, \emph{Quasi-uniform Topological Spaces}, Noordhoff, 1966.

\bibitem{Pervin} W.J. Pervin,  Connectedness in bitopological spaces, Indagationes Mathematicae (Proceedings) 70 (1967) 369-372.

\bibitem{PP2012} J. Picado, A. Pultr, \emph{Frames and Locales, Topology Without Points}, Birkh\"{a}user, 2012.

\bibitem{Priestley70} H.A. Priestley, Representation of distributive lattices by means of ordered Stone spaces, Bulletin of the London Mathematical Society 2 (1970)  186-190.

\bibitem{Priestley72} H.A. Priestley, Ordered topological spaces and the representation of distributive lattices, Proceedings of the London Mathematical Society 24 (1972) 507-530.
 
\bibitem{Reilly}I.L. Reilly, Zero-dimensional bitopological spaces, Indagationes Mathematicae (Proceedings) 76 (1973) 127-131.
  
\bibitem{Stone36} M.H. Stone, The theory of representation for Boolean algebras,  Transactions of the American Mathematical Society 40  (1936)  37–111.

\bibitem{Stone37a} M.H. Stone, Applications of the theory of Boolean rings to general topology, Transactions of the American Mathematical Society 41 (1937) 375-481.

\bibitem{Stone37b} M.H. Stone, Topological representation of distributive lattices and Brouwerian logics, \v{C}asopis pro p\v{e}stov\'{a}n\'{i} matematiky a fysiky 67 (1937) 1-25.
 

\bibitem{Swart} J. Swart, Total disconnectedness in bitopological spaces and product bitopological spaces, 
Indagationes Mathematicae (Proceedings) 74 (1971) 135--145.

\end{thebibliography}
\end{document}